\documentclass[10pt, a4paper, reqno]{amsart}
\numberwithin{equation}{section}
\usepackage{amsmath,amssymb,amsthm}
\usepackage{xcolor,textcomp}
\usepackage{graphicx, epstopdf}
\usepackage{braket,amsfonts}
\usepackage{dsfont}
\usepackage{mathtools}
\usepackage{mwe}

\usepackage[top=2.85cm,bottom=2.85cm,left=2.75cm,right=2.75cm]{geometry}
\usepackage{hyperref}
\hypersetup{colorlinks=true,linkcolor=red,citecolor=blue,filecolor=magenta,urlcolor=cyan} 
\usepackage[capitalise]{cleveref}
\usepackage{comment}
\usepackage{tikz}
\usepackage{pgfplots}
\usepackage{stmaryrd}
\usepackage[mathscr]{euscript}
\usepackage[sort,nocompress]{cite}

\newtheorem{theorem}{Theorem}[section]
\newtheorem{lemma}[theorem]{Lemma}
\newtheorem{remark}[theorem]{Remark}

\newtheorem{definition}[theorem]{Definition}

\usepackage{amsopn}
\DeclareMathOperator*{\esssup}{ess\,sup}

\DeclareMathOperator{\en}{en}

\newcommand\B{\mathcal{B}}

\newcommand\C{\mathcal{C}}
 
\newcommand\V{\mathbf{V}}
\newcommand\R{\mathbb{R}}

\newcommand\Ss{\mathbb{S}}

\newcommand\veps{\varepsilon}

\newcommand \LL{\mathcal{L}}
\newcommand{\boldvphi}{\boldsymbol{\vphi}}  
\newcommand{\boldphi}{\boldsymbol{\phi}}  

\newcommand{\disp}{\displaystyle}

\newcommand{\uvect}{\mathbf{u}}
\newcommand{\Yvect}{\mathbf{Y}}
\newcommand{\qvect}{\mathbf{q}} 
\newcommand{\Div}{\textnormal{div}}

\newcommand{\F}{\mathcal{F}}

\newcommand{\vphi}{\varphi}
\renewcommand{\rho}{\varrho}

\def\dx{\,\textnormal{d}x}
\def\dt{\textnormal{d}t}
\def\d{\,\textnormal{d}}

\makeatletter
\@namedef{subjclassname@2020}{%
	\textup{2020} Mathematics Subject Classification}
\makeatother

\title[Multicomponent reactive flows in moving domains]{On the Multicomponent Reactive flows in moving domains}

\author[Bhandari, Kra{\v{c}}mar,  Ne{\v{c}}asov{\'{a}}, Yang]{Kuntal Bhandari$^{\dagger}$ 
	\and  Stanislav Kra{\v{c}}mar$^{\ddagger}$
	\and \v{S}\'{a}rka Ne\v{c}asov\'{a}$^{\S, \star}$
 \and Minsuk Yang$^{\P}$}
 \thanks{$^{\dagger}$Institute of Mathematics,  Czech Academy of Sciences, \v{Z}itn\'a 25, 11567 Praha 1, Czech Republic; \texttt{e-mail: bhandari@math.cas.cz}.}
\thanks{$^{\ddagger}$Department of Technical Mathematics, Faculty of Mechanical Engineering, Czech Technical University,  Resslova 307/9, 12000 Praha 2, Czech Republic; 
\texttt{e-mail: stanislav.kracmar@fs.cvut.cz}. }
\thanks{$^{\S}$Institute of Mathematics,  Czech Academy of Sciences, \v{Z}itn\'a 25, 11567 Praha 1, Czech Republic; \texttt{e-mail: matus@math.cas.cz}. }
\thanks{$^{\P}$Department of Mathematics, Yonsei University, Yonseiro 50, Seodaemungu, Seoul 03722, Republic of Korea; \texttt{e-mail: m.yang@yonsei.ac.kr}}
\thanks{$^{\star}$Corresponding author.}



\date{\today}

\begin{document}

\begin{abstract}
This paper is concerned with the existence of global-in-time weak solutions to the multicomponent reactive flows inside  a moving domain whose shape in time is prescribed.     
The flow is governed by the   3D compressible {\em Navier-Stokes-Fourier system} coupled with the {\em equations of species mass fractions}.      The fluid velocity is supposed to fulfill  the complete slip boundary condition, whereas the heat flux and species diffusion fluxes satisfy the conservative boundary conditions.   The existence of weak solutions is obtained by means of suitable approximation techniques. To this end, we need to rigorously analyze the penalization of the boundary behavior, viscosity and the pressure in the weak formulation. 
\end{abstract}

\maketitle

 

	\section{Introduction}
	
\subsection{General setting and  problem statement} 
The multicomponent  reactive flows arise in many physical applications such as combustion, atmospherical modeling,
astrophysics, chemical reactions, etc.    In this regard,  we mention   the works of laminar flows \cite{smooke1982solution,dixon1984computer},  supersonic flames  \cite{drummond2005numerical,kailasanath1989numerical}, epitaxial growth  \cite{roenigk1987low,kee1987application},  the reentry problems into the earth's atmosphere \cite{Anderson1989hypersonic}, and the references therein.
 The physical background for the modeling of mixtures can be found in \cite{Pr}, an approach closer to modern understanding of continuum mechanics and thermodynamics  in \cite{RaTa}. Let us also mention the book \cite{Giovangigli}, where a general model of compressible chemically reacting mixtures under very general conditions is analyzed from the mathematical point of view, however, only in a small neighborhood of a given static solution.  For more details concerning modelling, we refer to the article by Bothe and Druet \cite{BD}.
 
In the present work, we are going to study the existence of weak solutions to the multicomponent reactive flows in a time-dependent domain. 
The balance equations  associated with  multicomponent  reactive flows express the {\em conservation of mass}, {\em momentum}, {\em energy} and {\em conservation of species mass}; see,  for instance, the book by Giovangigli \cite[Chapter 2.2]{Giovangigli}.  More precisely, the system under study is given by 
	\begin{align}
		&\partial_t \rho + \Div_x ( \rho \uvect) = 0  ,   \label{continuity-eq} 
		\\
		&\partial_t (\rho \uvect) + \Div_x (\rho [\uvect \otimes \uvect])  + \nabla_x p    = \Div_x \Ss 
		 ,  \label{momentum-eq} \\
		&\partial_t  (\rho e ) +  \Div_x(\rho e\uvect)  + 
		\Div_x \qvect + p \, \Div_x \uvect  = \Ss : \nabla_x \uvect + \Div_x \left(\sum_{k=1}^n h_k \F_k \right),   \label{temp-eq} 
		\\
		& \partial_t(\rho Y_k) + \Div_x(\rho Y_k \uvect) = \Div_x \F_k + \rho \sigma_k , \ \ k=1, ... , n     .           \label{pois-eq}   
	\end{align}  
The  model is characterized by the state variables:  total mass  density $\rho$,  mass average flow velocity $\uvect$, absolute temperature $\theta$ and  species mass fractions $Y_k$ for  $k=1,...,n$. Here, $p=p(\rho, \theta)$ is the pressure, $\Ss$ is the viscous stress tensor, $e=e(\rho, \theta, \Yvect)$ denotes the total energy per unit mass, $\qvect$ is the heat flux, $\F_k$ denotes the diffusion flux, $h_k$ is the enthalpy and 
$\sigma_k$  is the production rate  of the $k$-th species.  

Here and in the sequel, we  denote 
\begin{align*}
	\Yvect:= (Y_1, ..., Y_n) .
\end{align*}

\vspace*{.1cm}

Our analysis is based on the following physical assumptions, motivated by several numerical experiments related to the proposed model (see, for instance, \cite{Klein2001asymptotic} by Klein et al.):

\begin{enumerate} 
\item[1.]	We consider that the  fluid is Newtonian, meaning that the viscous stress tensor $\Ss$ satisfies
	\begin{align}\label{stress_tensor}
		\Ss(\theta, \nabla_x \uvect) = \mu(\theta) \left( \nabla_x \uvect + \nabla^\top_x \uvect -\frac{2}{3} \Div_x \uvect \mathbb I  \right) +\eta(\theta) \Div_x \uvect \mathbb I
	\end{align}
	with shear viscosity coefficient $\mu>0$ and bulk viscosity coefficient $\eta\geq 0$. 
	
\smallskip

\item[2.] The heat flux $\qvect:=\qvect(\theta, \nabla_x \theta)$ is determined by  Fourier's law
\begin{align}\label{Fourier_law}
	\qvect (\theta, \nabla_x \theta)= - \kappa (\theta) \nabla_x \theta 
\end{align}
with the heat conductivity coefficient $\kappa>0$. 

\smallskip

\item[3.] The  diffusion fluxes of the species are determined through  Fick's empirical law
\begin{align}\label{Ficks_law}
	\F_k = \zeta_k(\theta) \nabla_x Y_k , \quad k=1, ... , n,
	\end{align}
where $\zeta_k>0$ stands for the diffusion coefficient of the $k$-th species. 

\smallskip

\item[4.] The species formation enthalpies  $h_k$ for  $k=1,...,n$ are supposed to be constants. 

\smallskip 

\item[5.] We hereby denote the {\em species mass density} by $\rho_k$ for each $k=1,...,n$, 
and  suppose that the model is consistent with the principle of mass conservation, that is 
\begin{align}\label{rho-1}
\rho = \sum_{k=1}^n \rho_k .
\end{align} 
Now,  in terms of  the  species mass density $\rho_k$ and total mass density $\rho$, the 
  the species mass fractions $Y_k$ can be interpreted as 
\begin{align*} 
Y_k = \frac{\rho_k}{\rho} , \ \ \  k=1,...,n. 
\end{align*}  
In other words, $\rho_k = \rho Y_k$
and   we have 
\begin{align}\label{bound-Y-k-1}
0\leq Y_k \leq 1 , 
\end{align} 
 for each $k=1,...,n$. Moreover, in accordance with the relation \eqref{rho-1}, one has 
\begin{align}\label{mass-conservation}
	 \sum_{k=1}^n Y_k = 1 . 
\end{align}
Finally, we mention that the species production rate $\sigma_k$ satisfies 
\begin{align}\label{production-conservation}
\sum_{k=1}^n \sigma_k = 0 ,
\end{align} 
 and the second law of thermodynamics requires the entropy production to be non-negative.
 
\end{enumerate}

		\vspace*{.2cm} 
	
Let us describe the moving domain where the set of equations \eqref{continuity-eq}--\eqref{pois-eq} are considered.  
We consider a regular domain $\Omega_0 \subset \mathbb R^3$ occupied by the fluid at the initial time $t=0$. Then, we investigate the domain $\Omega_\tau$ (at time $t=\tau$) described by a given velocity field $\V(t,x)$ with $t\geq 0$ and $x\in \mathbb R^3$. More specifically, when $\V$ is regular enough,  we solve the associated system of differential equations
\begin{align}\label{trajec-eq}
	\frac{d}{dt} \mathbf X (t,x) =  \V(t,x) , \quad t>0 , \quad \mathbf X(0,x)=x ,
\end{align}
and set 
\begin{align}\label{domains}
	\begin{dcases} 
		\Omega_\tau = \mathbf X(\tau, \Omega_0) , \quad  \Gamma_\tau : = \partial \Omega_\tau \, \text{ and} \\
		Q_\tau := \cup_{t\in (0,\tau)} \{t\} \times \Omega_t  \ \ \ \forall \tau \in (0, T]. 	
	\end{dcases}
\end{align}
We  assume that the volume of the domain cannot degenerate in time, meaning that 
\begin{align}\label{measurable-condition}
	\exists \, M_0>0 \ \text{ such that } \ |\Omega_\tau|\geq M_0 \quad \forall \tau \in [0,T].
\end{align}
Furthermore,  we make the following assumption for the velocity field $\V$, namely
 \begin{align}\label{condition-V}
    \Div_x \V = 0  \ \text{ in the neighborhood of } \Gamma_\tau  \quad \forall \tau \in [0,T].
 \end{align}
\begin{remark} 
The condition \eqref{condition-V} is not restrictive. 
In fact,  it has been indicated in  \cite[Remark 5.3]{Sarka-et-al-ZAMP} that for a general $\V \in  \C^1([0,T];\C^3_c(\mathbb R^3, \mathbb R^3))$, one 
can find $\mathbf w \in W^{1,\infty}(Q_T; \R^3)$ such that $(\V- \mathbf w)|_{\Gamma_\tau} = 0$ for all $\tau \in [0,T]$ and $\Div_x \mathbf w=0$ on some neighborhood of $\Gamma_\tau$;  see also \cite[Section 4.3.1]{Sarka-Kreml-Neustupa-Feireisl-Stebel}. 
\end{remark}




\vspace*{.2cm}


In the fixed domain, the existence of global-in-time solutions for the system \eqref{continuity-eq}--\eqref{pois-eq},  
supplemented with physically relevant constitutive relations, was established by Giovangigli \cite[Chapter 9, Theorem 9.4.1]{Giovangigli} on the condition that the initial data are sufficiently close
to an equilibrium state. A simplified version of the  same problem with large initial data was  addressed in \cite{Trivisa-ARMA,Trivisa-Acta-sinica} in the one-dimensional case  and with one irreversible chemical reaction. In the three-dimensional situation, the existence theory of  the corresponding initial-boundary
value problem has been studied in \cite{Trivisa-Donatelli-ARMA,Feireisl-Novtny-proceedings} still with one irreversible chemical reaction.\footnote{A fast irreversible reaction of type  
\begin{equation*} 
A + B \rightarrow C, 
\end{equation*} 
was investigated by Bothe and Pierre in the case of reaction-diffusion system \cite{BP}, where it appears between species $A$ and $B$ with similar concentrations. Purely numerical and experimental results, involving irreversible reactions, were obtained for example in \cite{bothe-kroeger-warnecke,bubbles1,bubbles2}. See also recent result by  Mucha et al. \cite{MNS}.} 
However, in most practical situations,  one has to deal with several
species  which encounter  completely reversible  reactions  (see, for
instance, Bose \cite[Chapter 6]{Bose}).
In this regard, we notably mention the  work \cite{Eduard-CPAA-2008} by Feireisl et al., where
 the global-in-time weak solutions to  three-dimensional multicomponent reactive flows  have been explored for large  data. 
Some related works can be found for instance in \cite{degenerate-parabolic, MPZ_SIMA}.
 
 Coming to the context  of  compressible fluid flow systems in time-varying domains, 
we first address the work \cite{Sarka-Kreml-Neustupa-Feireisl-Stebel} by Feireisl et al., where the existence of weak solutions of the barotropic compressible  Navier-Stokes systems has been treated on time-dependent domains (as prescribed in \eqref{trajec-eq}--\eqref{domains}). Their approach is based on the penalization of the boundary behavior, viscosity and pressure in the weak formulation.  Later on, the existence of weak solutions to  the full {\em Navier-Stokes-Fourier} (NSF) systems in the time-dependent domain have been analyzed by   Kreml et al. \cite{sarka-aneta-al-JMPA,Sarka-et-al-ZAMP}.  We also mention that the compressible micropolar fluid on a time-varying domain with slip boundary conditions was considered in \cite{HBZ-2022}. Furthermore, the local-in-time existence of strong solutions to the compressible Navier-Stokes system on  moving domains was addressed in \cite{KNP-20}. The global well-posedness of compressible Navier-Stokes equations on a moving domain in the $L^p$-$L^q$ frame was investigated in \cite{KNP-20-springer}.  Recently, the authors in  \cite{Macha-Muha-Necasova-Roy-Srjdan}  studied the existence of a weak solution to a nonlinear fluid-structure interaction model with heat exchange where the shell is governed
 by linear thermoelasticity equations and encompasses a time-dependent domain that is filled with a fluid governed by the full NSF system. 
 In this regard, we also mention the work \cite{Kalousek2023existence}, where the authors
 analyzed  a system governing the interaction between two compressible mutually non-interacting fluids and a shell of Koiter-type that actually leads to  a time-dependent  domain filled by the fluids. 
Last but not least, we mention that the existence of weak solutions to the 3D Navier-Stokes-Fourier-Poisson system in time-varying domain has been explored in \cite{bhandari2023weak} with nonhomogeneous temperature  on the boundary where the heat flux does  not necessarily  satisfy the conservative boundary conditions. 
 We also refer to monograph \cite {KMNPW-L} where the barotropic and full system in the case of moving domain is described including all details and also with the introduction of ``new'' approximation scheme in the case of full system (it means slightly different approximation scheme according to moving domain than in the work of Feireisl and Novotn\' y \cite{Feireisl-Novotny-book}.

\subsection{Constitutive relations}\label{Section-Constitutive}

Let us prescribe the constitutive relations  associated to the system \eqref{continuity-eq}--\eqref{pois-eq}.

\subsubsection{Gibb's equation}   In agreement with the basic principles of statistical mechanics (e.g. \cite{Gallavotti} by Gallavotti), the  pressure   $p(\rho, \theta)$, 	specific entropy  $s(\rho, \theta, \Yvect)$, the specific internal energy $e(\rho, \theta, \Yvect)$ and the species mass fractions $Y_k$, for $k=1,...,n$,  
 are interrelated through the  Gibb's equation 
\begin{align}\label{equ-gibbs}
	\theta D s  = D e + p D \left(\frac{1}{\rho}\right) - \sum_{k=1}^n g_k D Y_k ,
\end{align}
where   $D$ stands  for the total differential with respect to the state variables $\rho$, $\theta$ and $Y_k$,   
and 
\begin{align}\label{retaion-g_k}
	g_k = h_k - \theta s_k , \quad k=1, ... , n, 
\end{align} 
(see \cite[Chapter 2.6.1]{Giovangigli} by Giovangigli).  Furthermore,  the species formation enthalpies $h_k$ and entropies $s_k$ are taken to be constants.

\subsubsection{Transport coefficients}  
As it has been addressed in \cite{becker1966gasdynamik},  the viscosity of a gas is independent of the density. Indeed, 	we consider the viscosity coefficients $\mu$ and $\eta$ to be continuously differentiable functions of the temperature, namely $\mu, \eta \in \C^1([0,+\infty))$ with
\begin{equation}\label{hypo-mu}
\begin{aligned}
&0< \underline \mu (1+\theta) \leq \mu(\theta) \leq \overline \mu (1+\theta), \quad \sup_{\theta \in [0,+\infty)} |\mu^\prime(\theta)| \leq \overline m , \ \ \text{and} \\
&  0\leq  \underline \eta (1+\theta) \leq \eta(\theta) \leq \overline \eta (1+\theta) . 
\end{aligned}
  \end{equation} 
In the above, all the quantities  $\underline \mu$, $\overline \mu$, $\overline m$, $\underline \eta$, $\overline \eta$ are real and positive.

	\vspace*{.1cm} 
	
 The heat conductivity coefficient $\kappa$ in  Fourier's law \eqref{Fourier_law} is assumed to be divided into two parts   	
\begin{equation}\label{hypo-kappa}
	\begin{aligned}
	&	\kappa(\theta) = \kappa_M (\theta) + \kappa_{R}(\theta) , \ \ \kappa_M , \, \kappa_R  \in \C^1([0,+\infty)) \ \, \text{such that}   \\
	& 0< \underline{\kappa_R} (1+\theta^3) \leq \kappa_R(\theta) \leq  \overline{\kappa_R}(1+\theta^3)   \, \ \text{and}\\
  & 0 < \underline{\kappa_M}(1+\theta) \leq \kappa_M (\theta) \leq \overline{\kappa_M}(1+\theta) ,
	\end{aligned}
\end{equation}
where $\underline{\kappa_R}$, $\overline{\kappa_R}$,
$\underline{\kappa_M}$ and $\overline{\kappa_M}$ are positive real numbers. 
In the above, the extra part 
$\kappa_R$ is the coefficient related to the effect of radiation (see, for instance, \cite{Oxenius} by Oxenius).

 \vspace*{.1cm}

	Finally, in agreement with  \cite[Proposition 7.5.7, Chapter 7]{Giovangigli}, the species diffusion coefficients coincide for all $k=1,..., n$. In what follows, we consider 
	\begin{align*}
		\zeta_k = \zeta \quad  \forall k \in\{1,..., n \},
	\end{align*}
and by means of the axioms in \cite[Chapter 7]{Giovangigli}, we assume that $\zeta\in \C^1([0,+\infty))$ such that 
\begin{align}\label{hypo-zeta}   
0< \underline{\zeta} \leq \zeta(\theta) \leq \overline \zeta (1+\theta) ,
\end{align}
with some  positive real numbers
 $\underline \zeta$ and $\overline \zeta$.

\subsubsection{Thermal equation of state}
  Motivated by Klein et al. \cite{Klein2001asymptotic},   	we consider that the pressure $p(\rho,\theta)$ is independent of the species concentrations. More precisely,   we consider 
	\begin{align}\label{hypo-press}
		p(\rho, \theta) = p_M(\rho, \theta) + p_R(\theta)  \quad \text{with } \, p_R(\theta) = \frac{a}{3} \theta^4 , \quad a>0 ,
	\end{align} 
where $p_R$ is the radiation part of the pressure $p$, which is particularly relevant in the high temperature regimes typical for combustion processes (see for instance, Bose \cite[Chapter 11.1]{Bose}), and $a$ is the so-called Stefan-Boltzmann constant. 

\vspace*{.1cm}

 Next, we consider  the internal energy $e(\rho, \theta, \Yvect)$  as 
	\begin{align}\label{hypo-energy}
		e(\rho, \theta, \Yvect) = e_M(\rho,\theta)    + e_R (\rho, \theta)   + \sum_{k=1}^n h_k Y_k     \quad \text{with } \, e_R (\rho, \theta) = \frac{a}{\rho}  \theta^4  ,
	\end{align}
and,  in accordance with  Gibb's law \eqref{equ-gibbs} and \eqref{hypo-energy}, we set 
	\begin{align}\label{hypo-entropy}
		s(\rho, \theta, \Yvect) = s_M(\rho,\theta) + s_R(\rho, \theta)   + \sum_{k=1} s_k Y_k  \quad \text{with }\, s_R(\rho, \theta) = \frac{4a}{3\rho} \theta^3  . 
	\end{align}

\vspace*{.2cm}
We further write  the following information. 
\begin{itemize} 
\item[--]	According to the hypothesis of thermodynamic stability, the molecular components $p_M$ and $e_M$  satisfy 
	\begin{align}\label{hypo-p_m}
		\frac{\partial p_M}{\partial \rho}  >  0 \quad \forall  \rho, \theta >0  , 
		\end{align}
	and there exists  some positive constant $c>0$ such that
	\begin{align}\label{hypo-e_m}
		0 < \frac{\partial e_M}{\partial \theta} \leq c    \quad \forall  \rho, \theta > 0  .
	\end{align}
 The latter one indicates that the specific heat at constant volume is uniformly bounded; we refer \cite{Eduard-CPDE} for more details.
 
	Moreover,   we have 
	\begin{align}\label{hypo-limit-e_m}
		\lim_{\theta \to 0^+} e_M(\rho, \theta) = \underline{e_M}(\rho) >0 \quad \text{for any fixed }  \rho >0
	\end{align}
	(see e.g. M\"{u}ller and Ruggeri \cite{Muller2013rational}), and 
		\begin{align}\label{hypo-bound_deri_e_m}
		\left| \rho \frac{\partial e_M}{\partial \rho}(\rho, \theta)\right| \leq c {e_M}(\rho, \theta) \quad \forall \rho, \theta >0.
	\end{align}

\vspace*{.1cm} 

\item[--] Motivated from the discussions in \cite[Chapter 1.4]{Feireisl-Novotny-book}, we suppose that there is a function 
\begin{align}\label{hypo_P}
P\in \C^1([0,\infty)), \ \ \ P(0) =0 , \ \ \ P^\prime(0) >0, 
\end{align}
and two positive constants $\underline Z$, $\overline Z$ such that 
\begin{align}\label{hypo-p_m_with_P}
	p_M(\rho, \theta) = \theta^{\frac{5}{2}} P\left(\frac{\rho}{\theta^{\frac{3}{2}}}\right)  \ \text{whenever }  0< \frac{\rho}{\theta^{\frac{3}{2}}} \leq \underline Z  \ \ \text{or} \ \ \frac{\rho}{\theta^{\frac{3}{2}}} > \overline Z.
\end{align}
Furthermore, $p_M$ and $e_M$ are interrelated through    the  general caloric equation (see e.g. Eliezer et al. \cite{Ghatak-et-al}):
\begin{align}\label{hypo-p_m_with_P-2}
p_M (\rho , \theta) = \frac{2}{3} \rho e_M(\rho, \theta) \ \ \ \text{for } \frac{\rho}{\theta^{\frac{3}{2}}} > \overline{Z}, 
\end{align}
which is a characteristic for mixtures of mono-atomic gases.

\vspace*{.1cm} 
\item[--] 
In order to get more information, we need to exploit the specific structure of the internal energy function $e(\rho, \theta, \Yvect)$.  In fact, we   prove that  
\begin{align}\label{lower_bound_rho_e}
	\rho (e_M(\rho, \theta) + e_R(\rho,\theta)) \geq a \theta^4 + \frac{3}{2} p_\infty \rho^{\frac{5}{3}},
\end{align}
where $p_\infty>0$ is given by \eqref{fact-2-1}. 
This can be shown in the following way. First, observe that 
\begin{align}\label{limit-e_M} 
\lim_{\theta\to 0^+} e_M(\rho, \theta) = \frac{3}{2\rho} \lim_{\theta\to 0^+} p_M(\rho,\theta) = \frac{3}{2\rho} \lim_{\theta\to 0^+} \theta^{\frac{5}{2}} P\left( \frac{\rho}{\theta^{\frac{3}{2}}}\right)  
=
\frac{3}{2}\rho^{\frac{2}{3}} \lim_{\theta\to 0^+}\frac{P(Z)}{Z^{\frac{5}{3}}} 
\end{align} 
(with $\disp Z=\frac{\rho}{\theta^{\frac{3}{2}}}$), where, in accordance with \eqref{hypo-limit-e_m},  we infer that
\begin{align}\label{fact-2-1}
	\lim_{Z\to \infty} \frac{P(Z)}{Z^{\frac{5}{3}}} = p_\infty > 0 . 
\end{align}
Beside that,  $e_M$ is a strictly increasing function of $\theta$  in $(0,\infty)$ for any fixed $\rho$; see \eqref{hypo-e_m}. This, together with the form of $e_R(\rho, \theta)$ in \eqref{hypo-energy}, and \eqref{limit-e_M} yields the required estimate \eqref{lower_bound_rho_e}.

%


\vspace*{.1cm} 
\item[--] 
In agreement with Gibb's relation \eqref{equ-gibbs}, the entropy component $s_M$ satisfies
\begin{align}\label{Gibbs-1}
	\frac{\partial s_M}{\partial \theta} = \frac{1}{\theta} \frac{\partial e_M}{\partial \theta} \ \  \ \text{and} \ \ \  
	\frac{\partial s_M}{\partial \rho} = -\frac{1}{\rho^2} \frac{\partial p_M}{\partial \theta}. 
\end{align}
We further write 
$s_M(\rho, \theta) = S(Z)$ with $\disp Z = \frac{\rho}{\theta^{\frac{3}{2}}}$,
and accordingly, one can deduce that 
\begin{align}\label{def-s-m} 
S^\prime(Z) =  -\frac{3}{2}\frac{\frac{5}{3} P(Z) - Z P^\prime(Z)}{Z^2} < 0 
\end{align}
in the degenerate area $\rho > \overline Z \theta^{\frac{3}{2}}$.  

We finally  assume that the third law of thermodynamics is satisfied, that is,
\begin{align}\label{third-law}
	\lim_{Z\to \infty} S(Z) = 0  .
\end{align}
\end{itemize} 

\subsubsection{Species production rate}
 The species production rates $\sigma_k$ are continuous functions of the absolute temperature $\theta$ and the species mass fractions $Y_1, ..., Y_n$. For simplicity, we assume that
\begin{align}\label{coeff-sigma-k}
	-\underline{\sigma} \leq \sigma_k(\theta, Y_1, ... , Y_n) \leq \overline{\sigma} , \quad  \forall \theta\geq 0, \ \ 0\leq Y_k \leq 1.
	\end{align}  
 We further suppose 
\begin{align}\label{sum-sigma-k}
	\sum_{k=1}^n   \sigma_k = 0 \ \ \text{and} \ \ \sum_{k=1}^n g_k \sigma_k \leq 0 ,
\end{align} 
in agreement with \eqref{production-conservation} and the latter one is enforced by the second law of thermodynamics. 
 
Finally, we assume 
\begin{align}\label{cond-3}
	\sigma_k(Y_1,...,Y_n) \geq 0 \  \text{ whenever }  Y_k=0 .
\end{align}

\subsection{Boundary conditions}\label{Sec-boundary}

 The system \eqref{continuity-eq}--\eqref{pois-eq} has to be supplemented with a suitable set of boundary conditions in order to obtain a mathematically well-posed problem.    


Let us  impose the complete slip boundary conditions 
	\begin{align}\label{Navier-slip}
		(\Ss \mathbf n) \times \mathbf n  = 0  \ \text{ on } \Gamma_t  \text{ for each } t \in [0,T], 
	\end{align}
	along with  the so-called 
	impermeability condition given by 
\begin{align}\label{imperm}
(\uvect - \V)\cdot \mathbf n = 0  \  \text{ on } \Gamma_t  \text{ for each } t \in [0,T],
\end{align}
	where $\mathbf n$ stands for the outer normal vector on  $\Gamma_t$.

\vspace*{.1cm}

The concept of weak solutions introduced in this paper requires the energy heat flux to satisfying  the conservative boundary conditions 
\begin{align}
\qvect \cdot \mathbf{n} = 0   \ \text{ on } \Gamma_t \text{ for each } t \in [0,T] .
\end{align}

	\vspace*{.1cm} 
 
We further consider that  the species diffusion fluxes satisfy 
	\begin{align}\label{boundary-species}
		\F_k \cdot \mathbf n =0 \  \text{ on } \Gamma_t \text{ for each } t \in [0,T] \ \text{and} \  k=1,...,n.
	\end{align}

	\subsection{Initial conditions}\label{Sec-initial}
	Let us write the initial conditions for our system \eqref{continuity-eq}--\eqref{pois-eq}.  

	We suppose 
\begin{align}\label{initial_conditions}
		\rho(0, \cdot)= \rho_0 \in L^{\frac{5}{3}}(\Omega_0) ,    \quad (\rho \uvect)(0,\cdot) = (\rho \uvect)_0 \quad \text{in } \Omega_0,
	\end{align}
	where we assume that the fluid density is zero outside the domain $\Omega_0$, more precisely
	\begin{align}\label{condition-density}
		\rho_0 \geq 0 \ \ \text{in }  \Omega_0, \ \ \rho_0 \neq 0 \  \text{ and }   \rho_0 =0 \ \text{ in }  \mathbb R^3 \setminus \Omega_0.
	\end{align}
 
	We write $\theta(0,\cdot)=\theta_0$ in $\Omega_0$  with  $\underline{\theta}\leq \theta_0 \leq \overline{\theta}$ for some  positive constants $\underline \theta$, $\overline \theta$,  and 
	\begin{align}\label{initial-Y-0-1} 
		Y_{k,0} = Y_k(0,\cdot) \ \text{ in } \Omega_0 , \ \  k=1, ... , n \ \text{ with } \Yvect_0 = (Y_{1,0}, ... , Y_{n,0}) .
	\end{align}
 Moreover, $Y_{k,0}$ verifies 
	\begin{align}\label{initial-Y-0-2} 
		0\leq Y_{k,0}(x) \leq 1 ,  \ \   \sum_{k=1}^n Y_{k,0}(x) =1  \,  \text{ for a.a. } x \in \Omega_0,  \ \text{and }\nabla_x Y_{k,0} \cdot \mathbf n =0, \ \ k=1,...,n . 
	\end{align}
	
We further denote 
	\begin{align}
		(\rho s)_0 = \rho_0 s(\rho_0, \theta_0, \Yvect_0) \in L^1(\Omega_0)  , 
	\end{align}
	and  assume that
\begin{align}\label{initial_energy}
	\mathcal E_0 := 	\int_{\Omega_0} \left( \frac{|(\rho  \uvect)_0|^2}{2\rho_0} +  \rho_0 e(\rho_0,\theta_0, \Yvect_0)   
 \right) < + \infty .
\end{align}

 \vspace*{.2cm}
 
	Our goal is to establish the global-in-time existence of weak solutions to the full system  \eqref{continuity-eq}--\eqref{pois-eq} in the domain $Q_T$ with the constitutive relations in Section \ref{Section-Constitutive},   boundary conditions in Section \ref{Sec-boundary} and initial conditions in Section  \ref{Sec-initial}.
		


	\section{Weak formulations}\label{Section-weak-target}
	
	In this section, we shall prescribe the expected weak formulation for system \eqref{continuity-eq}--\eqref{pois-eq}.
Throughout the section, we assume that the density $\rho$ remains ``zero" outside the   fluid domain $Q_T$.
 In fact, we will eventually show that if the initial density $\rho_0$ vanishes outside $\Omega_0$, then the same phenomenon remains unchanged  for the density  $\rho$   outside $\Omega_t$ for any $t\in (0,T]$ (see Section \ref{Section:Solid-part}).

 \subsection{Continuity equation}	 It is convenient to consider the continuity equation in the whole physical space $\mathbb R^3$ provided the density is supposed to be zero outside the fluid domain $\Omega_t$ for each $t\in [0,T]$. Specifically, the weak formulation of the continuity equation \eqref{continuity-eq} is supposed to be 
\begin{align}\label{weak-continuity}
	-	\int_0^T \int_{\Omega_t} \left(\rho \partial_t \vphi +    \rho \uvect \cdot \nabla_x \vphi \right)  =  \int_{\Omega_0} \rho_0(\cdot) \vphi(0, \cdot)  
\end{align}
for any  test function $\vphi \in \C^1( [0,T]\times \mathbb R^3; \mathbb R)$ with $\vphi(T,\cdot)=0$.   Of course, we assume that $\rho \geq 0$ a.e. in $\mathbb R^3$.

Moreover, the equation \eqref{continuity-eq}  will be satisfied in the sense of renormalized solutions introduced by DiPerna and Lions \cite{Lions-Diperna}: 
\begin{align}\label{weak-con-ren}
	-	\int_0^T \int_{\Omega_t} \rho B(\rho)\left( \partial_t \vphi +     \uvect \cdot \nabla_x \vphi \right) + 
	\int_0^T \int_{\Omega_t} b(\rho) \Div_x \uvect \vphi 
	=  \int_{\Omega_0} \rho_0 B(\rho_0)\vphi(0, \cdot)
\end{align}
for any   test function $\vphi \in \C^1( [0,T]\times \mathbb R^3; \mathbb R)$ with $\vphi(T,\cdot)=0$, $b \in L^\infty \cap \C([0,+\infty))$ such that $b(0)=0$ and $\displaystyle B(\rho) = B(1)+ \int_{1}^\rho \frac{b(z)}{z^2}$.

\subsection{Momentum equation}   We now write the expected weak formulation for the momentum equation \eqref{momentum-eq}.
It can be identified as 
\begin{align}\label{weak_form_momen}
-\int_0^T \int_{\Omega_t} \left( \rho \uvect \cdot \partial_t  \boldvphi   +  \rho[\uvect \otimes \uvect] : \nabla_x  \boldvphi + p(\rho, \theta) \Div_x \boldvphi   \right)     - \int_{\Omega_0}   (\rho \uvect)_0\cdot  \boldvphi(0, \cdot) \notag \\
 =- \int_0^T \int_{\Omega_t} \Ss : \nabla_x \boldvphi 
\end{align}
for  any test function $\boldvphi \in \C^1(\overline Q_T; \mathbb R^3)$ satisfying 
\begin{align*}
\boldvphi(T,\cdot)=0 \ \ \text{in } \Omega_T \ \text{ and }  	\boldvphi \cdot \mathbf n |_{\Gamma_t} = 0 \quad \text{for  any $t \in [0,T]$.}
\end{align*}
	The impermeability condition will  then be satisfied in the sense of trace, 
	\begin{align*}
\uvect \in L^2(Q_T;  \mathbb R^3),  \ \ \nabla_x \uvect\in L^2(Q_T;  \mathbb R^{3\times 3})  , \ \ (	\uvect - \V)\cdot \mathbf n |_{\Gamma_t} = 0 \quad \text{for any } t\in [0,T].   
	\end{align*}

\subsection{Species mass conservation}  Let us introduce the concept of entropy solutions for \eqref{pois-eq}. We hereby recall that $\F_k=\zeta(\theta)\nabla_x Y_k$ and the weak formulation will be 
		\begin{align}\label{eq-mass-species-1}
			-\int_0^T \int_{\Omega_t} \left[\rho Y_k \partial_t \vphi + \rho Y_k \uvect \cdot \nabla_x \vphi  - \zeta(\theta) \nabla_x Y_k \cdot \nabla_x \vphi   \right] =&  \notag \\
			\int_0^T \int_{\Omega_t}  \rho \sigma_k \vphi + \int_{\Omega_0} \rho_0 Y_{k,0} \vphi(0,\cdot)  \quad &\text{ for  } k=1, ..., n  ,
		\end{align}
to be satisfied for any test function $\vphi \in \C^1(\overline{Q}_T;\mathbb R)$ with $\vphi\geq 0$ and $\vphi(T,\cdot)=0$ in $\Omega_T$, together with  
\begin{align}\label{eq-mass-specs-2} 
	-\int_0^T \partial_t  \psi \int_{\Omega_t}  \rho G(\Yvect) + & \int_0^T \psi \int_{\Omega_t} \sum_{k=1}^n\zeta(\theta) G_0 |\nabla_x Y_k|^2 
\notag \\
& \qquad \qquad 
	\leq  
	\int_0^T \psi \int_{\Omega_t}  \sum_{k=1}^n \rho \frac{\partial G(\Yvect)}{\partial Y_k} \sigma_k + \int_{\Omega_0} \rho_0 G(\Yvect_0) \psi(0) 
 \end{align}
for any $\psi\in \C^1_c([0,T))$ such that $\psi\geq 0$, $\partial_t\psi\leq 0$ and any convex function $G\in \C^2(\mathbb R^n;\mathbb R)$ verifying
	\begin{align}\label{condition-ellipic-G}
	\sum_{k,l=1}^n \frac{\partial^2 G(\Yvect)}{\partial Y_k \partial Y_l} \xi_k \xi_l \geq G_0 |\xi|^2
	\end{align}
	for any $\xi:=(\xi_1,..., \xi_n) \in \mathbb R^n$ and for some $G_0>0$.
	
Moreover, in agreement with \eqref{bound-Y-k-1}--\eqref{mass-conservation}, $Y_k$  should satisfy  
\begin{align}\label{boundedness-Y_k}
0 \leq Y_k \leq 1 \ \ \text{in } Q_T, \ \  \forall k=1,...,n , \ \text{ and }\,
   \sum_{k=1}^n Y_k = 1.
\end{align}

\subsection{Entropy production} 
It is  more convenient to replace \eqref{temp-eq} by the total entropy balance.  Indeed, one can obtain 
	\begin{align}\label{entropy_eq}
		\partial_t (\rho s) + \Div_x (\rho s \uvect ) + \Div_x \left(\frac{\qvect}{\theta} - \sum_{k=1}^n s_k \F_k \right) = \frac{1}{\theta} \left( \Ss : \nabla_x \uvect - \frac{\qvect\cdot \nabla_x \theta}{\theta} - \sum_{k=1}^n \rho g_k \sigma_k    \right) ,
	\end{align}
cf. \cite[Chapter 2.6.2, 2.6.3]{Giovangigli}. 	 
	
Here, we  mention  that by virtue of the second law of thermodynamics, the entropy production rate 
	\begin{align}\label{entr-prod} 
		\vartheta_{\en}  =  \frac{1}{\theta} \left( \Ss : \nabla_x \uvect - \frac{\qvect\cdot \nabla_x \theta}{\theta} - \sum_{k=1}^n \rho g_k \sigma_k    \right) 
	\end{align} 
must be non-negative for any admissible process. In particular, the species production rates $\sigma_k$ for $k=1,...,n$ have  to satisfy \eqref{sum-sigma-k}.  

\vspace*{.1cm}

Following the strategy of  \cite{Feireisl-Novtny-proceedings}, we replace the equation \eqref{entropy_eq}  by the following entropy inequality: 
\begin{align}\label{entropy_ineq}
  \partial_t (\rho s) + \Div_x (\rho s \uvect ) + \Div_x \left(\frac{\qvect}{\theta} - \sum_{k=1}^n s_k \F_k \right) \geq \frac{1}{\theta} \left( \Ss : \nabla_x \uvect - \frac{\qvect\cdot \nabla_x \theta}{\theta} - \sum_{k=1}^n \rho g_k \sigma_k    \right) .
\end{align} 
Then the weak formulation  for \eqref{entropy_ineq} will be of the form:
\begin{multline}\label{weak-for-entropy}
		- \int_0^T \int_{\Omega_t} \left[\rho s\partial_t \vphi + \rho s\uvect \cdot \nabla_x \vphi + \left(\frac{\qvect}{\theta}-\sum_{k=1}^n s_k \F_k\right) \cdot \nabla_x \vphi \right]  - \int_{\Omega_0} (\rho s)_0 \vphi(0,\cdot)   \\
		\geq 
		\int_0^T\int_{\Omega_t} \frac{\vphi}{\theta}  \left( \Ss : \nabla_x \uvect  - \frac{\qvect\cdot \nabla_x \theta}{\theta} - \sum_{k=1}^n \rho g_k \sigma_k \right) 
\end{multline} 
	for any  test function $\vphi \in \C^1(\overline Q_T; \mathbb R)$ with $\vphi \geq 0$ and $\vphi(T,\cdot)=0$ in $\Omega_T$.

	Accordingly,   the entropy production $\vartheta_{\en}$ associated to a weak solution satisfies 
	 	\begin{align}
	 	\vartheta_{\en}  \geq   \frac{1}{\theta} \left( \Ss : \nabla_x \uvect - \frac{\qvect\cdot \nabla_x \theta}{\theta} - \sum_{k=1}^n \rho g_k \sigma_k    \right) 
	 \end{align} 
	rather than \eqref{entr-prod}. Observe that $\vartheta_{\en}>0$,  in agreement with the heat flux $\qvect=-\kappa(\theta) \nabla_x \theta$ (with $\kappa(\theta)>0$)  and the fact that  $\sum_{k=1}^n g_k \sigma_k \leq 0$ (see \eqref{sum-sigma-k}).

\subsection{Energy balance} Assuming all the quantities of concern are smooth and multiplying the momentum equation \eqref{momentum-eq} by $(\uvect -  \V)$, then integrating by parts w.r.t. the space variable and integrating the internal energy equation \eqref{temp-eq}, 
one has (in agreement with the boundary conditions and  the continuity equation)
	\begin{align}\label{energy-1}
	&	 \frac{\d}{\dt} \int_{\Omega_t}\left(\frac{1}{2} \rho |\uvect|^2  + \rho e_M(\rho, \theta) + \rho e_R(\rho, \theta) + \sum_{k=1}^n \rho h_k Y_k   \right) 
		  \notag \\
		  = &
	-\int_{\Omega_t} \big(\rho[\uvect \otimes \uvect]    : \nabla_x \V - \Ss : \nabla_x \V + p(\rho, \theta) \Div_x \V \big)
	+ \int_{\Omega_t} \partial_t(\rho \uvect) \cdot \V \ \ \ \forall t\in [0,T].
	\end{align}
From  the equation of species mass fraction \eqref{pois-eq}, one has 
\begin{align*}
	 \frac{\d}{\dt} \int_{\Omega_t} \sum_{k=1}^n \rho h_k Y_k  & = 
	  \int_{\Omega_t} \sum_{k=1}^n h_k \big[ -\Div_x (\rho Y_k \uvect) + \Div_x (\F_k) + \rho \sigma_k \big] \\
	  & =    \int_{\Omega_t} \sum_{k=1}^n   \rho  h_k \sigma_k 
\end{align*}
and consequently,  the equality \eqref{energy-1} reduces to 
  	\begin{align}\label{energy-2}
  	&	 \frac{\d}{\dt} \int_{\Omega_t}\left(\frac{1}{2} \rho |\uvect|^2  + \rho e_M(\rho, \theta) + \rho e_R(\rho, \theta)    \right)  
  		 \notag \\
  		= &
  		-\int_{\Omega_t} \big(\rho[\uvect \otimes \uvect]    : \nabla_x \V - \Ss : \nabla_x \V + p(\rho,\theta) \Div_x \V \big)
  		+ \int_{\Omega_t} \partial_t(\rho \uvect) \cdot \V - \int_{\Omega_t} \sum_{k=1}^n \rho h_k \sigma_k
  \end{align}
  for any $t\in [0,T]$.

Writing in a  formal way, we will have 
\begin{align}\label{energy-formal-main}
&	-	\int_0^T \partial_t \psi \int_{\Omega_t}  \left(\frac{1}{2} \rho |\uvect|^2  + \rho e_M(\rho,\theta)  + \rho e_{R}(\rho, \theta) 
\right) \notag \\
	 =&	-\int_0^T  \psi \int_{\Omega_t}  \big(\rho[\uvect \otimes \uvect]    : \nabla_x \V  - \Ss : \nabla_x \V  + p(\rho,\theta) \Div_x \V  \big) 
		+ \int_0^T \int_{\Omega_t} \partial_t(\rho \uvect) \cdot \V  \psi  
		\notag \\ 
 &  -\int_0^T \psi \int_{\Omega_t}  \sum_{k=1}^n \rho h_k \sigma_k +\psi(0) \int_{\Omega_0} 	  \left(\frac{1}{2} \frac{|(\rho \uvect)_{0}|^2}{\rho_{0}}  + \rho_{0} \left(e_M + e_{R,\xi}\right)(\rho_{0},\theta_{0})  
 \right)
	\end{align}
for any $\psi \in \C^1_c([0,T))$ such that $\psi\geq 0$ and $\partial_t \psi\leq 0$.

\section{Main result}   

\begin{definition}[Weak solution]\label{def}
We say that the quantity 
$(\rho,\uvect,\theta, \Yvect)$ is a weak solution  of the problem 
\eqref{continuity-eq}--\eqref{pois-eq} with the constitutive relations in Section \ref{Section-Constitutive}, boundary
	conditions \eqref{Navier-slip}--\eqref{boundary-species}
	and initial conditions \eqref{initial_conditions}--\eqref{initial_energy}
	 if the following items hold:
	\begin{itemize}
		\item 
		$\rho\in L^\infty(0,T;L^{\frac{5}{3}}(\mathbb{R}^3))$, $\rho\geq 0$, $\rho\in L^q(Q_T)$ with some certain $q>1$,
		
		\item $\uvect, \nabla_x\uvect \in L^2(Q_T)$, 
		$\rho \uvect \in L^\infty(0,T; L^1(\mathbb R^3))$,
		
		\item $\theta>0$ a.e. on $Q_T$, 
		$\theta \in L^\infty(0,T;L^4(\mathbb{R}^3))$,  $\theta, \nabla_x\theta, \log\theta, \nabla_x\log\theta \in L^2(Q_T)$,

  \item   $\rho s, \, \rho s \uvect  \in L^1(Q_T)$,
  
  \item  $Y_k \in L^\infty(Q_T),  \, \nabla_x Y_k \in L^2(Q_T)$ for each $1\leq k \leq n$,  and
  
\item the relations  \eqref{weak-con-ren}, \eqref{weak_form_momen}, \eqref{eq-mass-species-1}--\eqref{boundedness-Y_k}, \eqref{weak-for-entropy} and \eqref{energy-formal-main} are satisfied.
\end{itemize}
\end{definition}

Let us state the main result of the present paper.

\begin{theorem}[Main result]\label{Main-theorem}
Assume that $\Omega_0 \subset \mathbb{R}^3$ is a bounded domain of the class $\C^{2+\nu_0}$, $\nu_0>0$, and suppose that $\V\in \C^1([0,T]; \C^3_c(\mathbb{R}^3;\mathbb{R}^3))$ verifying \eqref{trajec-eq},  \eqref{condition-V}, and the hypotheses in Section \ref{Section-Constitutive} are satisfied. Then the multicomponent reactive flows given by  \eqref{continuity-eq}--\eqref{pois-eq} with boundary conditions in Section \ref{Sec-boundary} and initial conditions in Section \ref{Sec-initial} admits a weak solution in the sense of Definition \ref{def} on any finite time interval $(0,T)$.
\end{theorem}

\section{Penalization technique}\label{sec-penalized}
 

Let us choose $R>0$ such that 
\begin{align*}
\V|_{[0,T]\times \{ |x|>R \}} = 0 , \quad \overline \Omega_0 \subset \{ x \in \mathbb R^3 : |x| \leq  R \},
\end{align*}
and then take the reference domain $\B:= \{x\in \mathbb R^3 : |x| < 2R\}$.

\subsection{Mollification of the coefficients and initial data}\label{Section-hypo-molified}

We need to extend the following coefficients to the space-time  reference domain $(0,T)\times \B$. 

\smallskip 

\noindent
$\bullet$
The viscosity coefficients and species diffusion coefficient  are taken as 
\begin{align}\label{def-mu-omega}
& \mu_\omega (\theta) = f_\omega \mu(\theta) \in \C^\infty_c([0,T]\times \mathbb R^3) ,\\
\label{def-nu-omega}
& \eta_\omega (\theta) =  f_\omega\eta(\theta) \in \C^\infty_c([0,T]\times \mathbb R^3) \text{ and} \\
\label{def-sigma-omega}
&\zeta_\omega (\theta) = f_\omega \zeta( \theta) \in \C^\infty_c([0,T]\times \mathbb R^3)
\end{align}
with some function $f_\omega \in \C^\infty_c([0,T]\times \mathbb R^3)$  such that 
\begin{align}\label{choice-f-omega}
	\begin{dcases}
 0< \omega  \leq f_\omega \leq 1 \ \ \ \text{in } \, [0,T] \times \B ,	\\
f_\omega(t,\cdot) |_{\Omega_t} = 1 \ \text{ for each } t \in [0,T] \text{ and}\\
	 \|f_\omega \|_{L^{q}(((0,T) \times \B)\setminus Q_T) } \leq c \, \omega  \ \, \text{for some  $c>0$ and  $q\geq \frac{5}{3}$.} 
	\end{dcases}
\end{align}
From the above definitions, it is clear that 
\begin{align}
\mu_\omega , \  \eta_\omega , \ \zeta_\omega \to 0 \ \,  \text{ a.e. in } \,  ((0,T)\times \B )\setminus Q_T  \, \text{ as } \, \omega \to 0.
\end{align}

\noindent
$\bullet$
We also introduce  variable heat conductivity coefficient as follows:
\begin{align}\label{defi-kappa-nu}
	\kappa_\nu (\theta) = \chi_\nu  \kappa(\theta) , \ \text{ where } \chi_\nu = 1 \text{ in } Q_T \text{ and } \chi_\nu = \nu \text{ in } ((0,T)\times \B) \setminus Q_T.  
\end{align}

\noindent
$\bullet$
Similarly, we introduce a variable coefficient $a_\xi:= a_\xi(t,x)$ representing the radiative parts of the pressure, internal energy and entropy, given by 
\begin{align}\label{a-xi}
		a_\xi  = \chi_\xi  a , \ \text{ where } \chi_\xi = 1 \text{ in } Q_T \text{ and } \chi_\xi = \xi \text{ in } ((0,T)\times \B) \setminus Q_T.
\end{align} 
   We now set 
\begin{align}\label{pressure-pena}
&	p_{\xi, \delta} (\rho, \theta) = p_M(\rho, \theta) + \frac{a_\xi}{3} \theta^4  + \delta \rho^\beta  \quad \text{for } \beta \geq  4,  \ \delta>0 ,\\
	\label{energy-pena}
	& e_\xi(\rho, \theta, \Yvect) = e_M(\rho,\theta) + e_{R,\xi}(\rho, \theta)
	 + \sum_{k=1}^n h_k Y_k  \quad \text{with } \, e_{R,\xi}(\rho, \theta)
	 = \frac{a_\xi \theta^4}{\rho}, \text{ and}
	\\
	\label{entropy-pena} 
	 & s_\xi(\rho, \theta, \Yvect) = s_M(\rho, \theta)  + s_{R,\xi}(\rho, \theta)    + \sum_{k=1}^n s_k Y_k \quad \text{with } \, s_{R,\xi}(\rho, \theta) = \frac{4a_\xi \theta^3}{3\rho}   .
\end{align}


\smallskip
\noindent 
$\bullet$
Let us now define the modified initial data $\rho_{0,\delta}$, $(\rho \uvect)_{0,\delta}$, $\theta_{0,\delta}$ and $Y_{k,0, \delta}$, $k=1,...,n$.  
We consider $\rho_{0,\delta}$ such that
\begin{subequations}\label{pena-rho-initial}
\begin{align}
	\rho_{0,\delta} \geq 0, \ \ \rho_{0,\delta} \neq 0 \  \text{ in } \Omega_0, \ \ \rho_{0,\delta} =0 \ \text{ in }  \mathbb R^3\setminus \Omega_0, \ \ \int_{\B}\left(\rho^{\frac{5}{3}}_{0,\delta} + \delta \rho^{\beta}_{0,\delta} \right) \leq c , 
	\\ 
	\text{and } \
	\rho_{0,\delta} \to \rho_0 \ \text{ in } L^{\frac{5}{3}}(\B) \ \text{ as } \delta \to 0 , \ \ \ |\{ \rho_{0,\delta} <\rho_0 \}| \to  0 \ \text{ as } \delta \to 0. 
\end{align}
\end{subequations}

Next, the penalized initial data for the momentum part is taken in such a way that
\begin{align}
	(\rho \uvect)_{0,\delta} = \begin{dcases}
		(\rho \uvect)_0  \ &\text{if } \rho_{0,\delta} \geq \rho_0 , \\
		0                \       & \text{else} .
	\end{dcases}
\end{align} 

For the temperature part, we consider $0<\underline \theta \leq \theta_{0,\delta}\leq \overline \theta$,
 where  
$\theta_{0,\delta} \in L^\infty(\B) \cap \C^{2+\nu_0}(\B)$ with some exponent $\nu_0\in (0,1)$ where $\underline \theta$, $\overline \theta$ are positive real numbers as mentioned earlier. 

The perturbed initial data for the species mass fraction is verifying $Y_{k,0,\delta}\in \C^\infty(\B)$ with
\begin{equation}\label{modified-mass-initial}  
\begin{aligned}
	0\leq Y_{k,0,\delta}(x) \leq 1,  \ \ \sum_{k=1}^n Y_{k,0,\delta}(x) =  1  \ \ \text{for a.a. } x\in \Omega_0 \ \  \text{and } \nabla_x Y_{k,0,\delta}\cdot \mathbf n =0, \ \ k=1,...,n .
\end{aligned}
\end{equation}


Moreover, $\rho_{0,\delta}$, $\theta_{0,\delta}$ and $\Yvect_{0,\delta}:=(Y_{1,0,\delta}, ... , Y_{n,0,\delta}$) are considered in  such a way that 
\begin{align}
\int_{\Omega_0} \rho_{0,\delta}( e_M + e_{R, \xi})(\rho_{0,\delta}, \theta_{0,\delta}) \to  	\int_{\Omega_0} \rho_{0} (e_M + e_{R, \xi})(\rho_{0}, \theta_{0}),
\end{align}
\begin{align}
	\rho_{0,\delta} (s_M + s_{R, \xi})(\rho_{0,\delta}, \theta_{0,\delta}) &\to   \rho_{0} (s_M + s_{R, \xi})(\rho_{0}, \theta_{0}) \ \text{ weakly in }  L^1(\Omega_0)  , \\
 \label{pena-Y-initial} 
\text{and} \ \	 \rho_{0, \delta} Y_{k,0,\delta} & \to \rho_0 Y_{k,0} \ \text{ weakly in }  L^1(\Omega_0) ,
\end{align}
as $\delta \to 0$. 


\subsection{Strategy of the proof}


Below we   briefly  describe the strategy of the proof for \Cref{Main-theorem}.
\begin{enumerate}
\item[1.]
Motivated by the work of  Stokes and Carey in \cite{StoCar},    we add the penalized term 
 \begin{align}\label{pena-boundary-mom} 
\frac{1}{\veps} \int_0^T \int_{\Gamma_t}  (\uvect - \V ) \cdot \mathbf n \ \vphi \cdot \mathbf n \ \ \text{(for  $\veps > 0$ small)} 
\end{align} 
to  the momentum equation. 
In principle, this allows us to take care of the slip boundary conditions. Indeed, after obtaining  some uniform estimates w.r.t. $\veps$,  the  additional term \eqref{pena-boundary-mom} yields the boundary condition $(\uvect - \V)\cdot \mathbf n = 0$
on $\Gamma_t$ as $\veps\to 0$.    Accordingly, the reference domain $(0,T) \times \B$ is separated by an impermeable interface $\cup_{t\in (0,T)} \{t\} \times \Gamma_t$ to a {\em fluid domain} $Q_T$ and a {\em solid domain} $((0,T)\times \B) \setminus Q_T$. 
As a result, we need to deal with  the behavior of the solution in the solid domain. This is why we consider the variable coefficients $\mu_\omega$, $\eta_\omega$, $\zeta_\omega$, $\kappa_\nu$, $a_\xi$ as presented in Section \ref{Section-hypo-molified}. Moreover, similar to the existence theory developed in   \cite{Feireisl-Novotny-book}, we  introduce the artificial pressure $p_{\xi,\delta}$ with an extra term $\delta \rho^\beta$ (see \eqref{pressure-pena}), which gives  some  more  (regularity) information of the density.

\vspace*{.1cm}

\item[2.] We  add a term $\lambda \theta^5$ to the energy balance and $\lambda \theta^{4}$ to the entropy balance with  $\lambda >0$. These terms yield a control over the temperature in the solid domain. More  precisely, these extra penalized terms help to get rid of some unusual terms in  the solid domain due to the higher power of $\theta$.

\vspace*{.1cm}

\item[3.] Keeping $\veps, \omega, \nu, \lambda, \xi$ 
 and $\delta>0$ fixed, 
we use the existence theory for the multicomponent reactive flows   
in the fixed reference domain developed in \cite{Eduard-CPAA-2008}.  

\vspace*{.1cm}

\item[4.]  Taking the initial density
$\rho_0$ vanishing outside $\Omega_0$ and letting $\veps \to 0$
for fixed $\omega, \nu, \lambda, \xi, \delta > 0$,
 we obtain a  ``two-fluid''
system where the density vanishes in the solid part $\left((0,T)
\times \B \right) \setminus Q_T$. 
 Then, in order to get rid of the terms in $\left((0,T)
 \times \B \right) \setminus Q_T$, we make
  all other parameters tend to zero.  To this end,   proper scaling is required  to let the parameters 
  $\omega, \nu, \xi, \lambda$ go 
   to zero simultaneously. This has been rigorously prescribed in Section \ref{Section:Solid-part}. Finally,  we let $\delta\to 0$ in a standard fashion, as already been used in other related works for compressible fluids.
\end{enumerate}

\section{Weak formulations in the reference domain}

Let us  consider the set of equations 
\eqref{continuity-eq}--\eqref{pois-eq} in the reference domain $(0,T)\times \B$
and state the weak formulation of this  penalized problem.   We consider the velocity $\uvect$  vanishes on the boundary  $(0,T)\times \partial \B$, that is 
\begin{align}\label{boundary-extended-1}
\uvect|_{\partial \B}  =0 \  \text{ for each } t\in [0,T]  .
\end{align}
Furthermore, we assume that the heat flux $\qvect$ and species diffusion fluxes $\F_k$ satisfy 
\begin{align}\label{boundary-extended-2}
 \qvect \cdot \mathbf n = 0  \ \text{ and } \ \F_k \cdot \mathbf n  = 0 \ \text{ on } (0,T)\times \partial \B.    
\end{align}

\subsection{Continuity equation}  The weak formulation for the continuity equation reads as 
\begin{align}\label{weak-conti}
		-\int_0^T \int_{\B} \rho B(\rho)\left( \partial_t \vphi +    \uvect \cdot \nabla_x \vphi \right) 
+	\int_0^T \int_{\B} b(\rho) \Div_x \uvect \vphi 
	=  \int_{\B} \rho_{0,\delta} B(\rho_{0,\delta})\vphi(0, \cdot)
\end{align}
for any  test function $\vphi \in \C^1_c( [0,T)\times \B; \mathbb R)$ and any $b \in L^\infty \cap \C([0,+\infty))$
 such that $b(0)=0$ and $\displaystyle B(\rho) = B(1)+ \int_{1}^\rho \frac{b(z)}{z^2}$.

\subsection{Momentum equation}
The momentum equation is now represented by the following family of integral identities 
\begin{equation}
\begin{aligned}\label{weak-momen}
-	\int_0^T \int_{\B} \big( \rho \uvect \cdot \partial_t  \boldvphi   +  \rho[\uvect \otimes \uvect] : \nabla_x  \boldvphi + p_{\xi, \delta}(\rho, \theta) \Div_x \boldvphi   \big)   
 +\int_0^T \int_{\B} \Ss_\omega : \nabla_x \boldvphi \\
  +  \frac{1}{\veps} \int_0^T \int_{\Gamma_t}  (\uvect -\V) \cdot \mathbf n \ \boldvphi \cdot \mathbf n 
 =
    \int_{\B}   (\rho \uvect)_{0,\delta}\cdot  \boldvphi(0, \cdot) 
\end{aligned}
\end{equation}
	for  any  test function $\boldvphi\in \C^1_c([0,T) \times \B; \mathbb R^3)$,
	  where 
\begin{align}\label{stress_tensor-omega}
	\Ss_\omega(\theta, \nabla_x \uvect) = \mu_\omega(\theta, t,x) \left( \nabla_x \uvect + \nabla^\top_x \uvect -\frac{2}{3} \Div_x \uvect \mathbb I  \right) +\eta_\omega(\theta,t,x) \Div_x \uvect \mathbb I.
\end{align}

\subsection{Species mass conservation}  
 The weak formulation for the species mass fractions in the reference domain looks like 
\begin{align}\label{eq-mass-species-1-pena}
		-\int_0^T \int_{\B} \big[\rho Y_k \left(\partial_t \vphi +  \uvect \cdot \nabla_x \vphi \right)  - \zeta_\omega(\theta) \nabla_x Y_k \cdot \nabla_x \vphi   \big] =& \notag \\
		\int_0^T \int_{\B}  \rho \sigma_k \vphi + \int_{\B} \rho_{0,\delta} Y_{k,0, \delta} \vphi(0,\cdot)  \quad &\text{ for } \ k=1, ..., n  ,
\end{align} 
which is to be satisfied for any test function $\vphi \in \C^1([0,T)\times \B   ;\mathbb R)$, $\vphi\geq 0$,  together with  
%
	\begin{align}\label{eq-mass-specs-2-pena} 
		-\int_0^T \partial_t  \psi \int_{\B}  \rho G(\Yvect) + & \int_0^T \psi \int_{\B} \sum_{k=1}^n\zeta_\omega(\theta) G_0 |\nabla_x Y_k|^2 
		\notag \\
		& \qquad \qquad 
		\leq  
		\int_0^T \psi \int_{\B}  \sum_{k=1}^n \rho \frac{\partial G(\Yvect)}{\partial Y_k} \sigma_k + \int_{\B} \rho_{0,\delta} G(\Yvect_{0,\delta}) \psi(0) 
	\end{align}
for any $\psi\in \C^1_c([0,T))$ such that  $\psi\geq 0$, $\partial_t\psi\leq 0$ and any convex function $G\in \C^2(\mathbb R^n;\mathbb R)$ verifying
\begin{align}\label{convex-func}
	\sum_{k,l=1}^n \frac{\partial^2 G(\Yvect)}{\partial Y_k \partial Y_l} \xi_k \xi_l \geq G_0 |\xi|^2
\end{align}
for any 
$\xi:=(\xi_1,..., \xi_n) \in \mathbb R^n$ and  some $G_0>0$.

Moreover, 
it should satisfy that 
\begin{align}\label{boundedness-Y_k-pena}
	0\leq Y_k \leq 1 \ \,  \text{in  $(0,T)\times \B$  for each} \, \  k=1,...,n , \ \ \text{and }  \sum_{k=1}^n Y_k = 1 ,
\end{align}
and 
\begin{align}\label{mass-conservation-B}
\sum_{k=1}^n \sigma_k =0 .
\end{align} 
This is in accordance with the assumption \eqref{mass-conservation} and the fact that the species production rate $\sigma_k$ (for each $k=1,...,n$) outside $Q_T$ is taken to be zero. 

\subsection{Entropy production} We now write the 
 penalized entropy inequality as follows:
\begin{align} \label{penalized-entropy-1} 
	- & \int_0^T \int_{\B} \left(\rho s_\xi \left(\partial_t \vphi +  \uvect \cdot \nabla_x \vphi \right) - \frac{\kappa_\nu(\theta)}{\theta} \nabla_x \theta \cdot \nabla_x \vphi   - \sum_{k=1}^n s_k  \zeta_{\omega}(\theta) \nabla_x Y_k        \cdot \nabla_x \vphi \right)  \notag \\
	 & \qquad  - \int_{\B} (\rho s)_{0,\delta} \vphi(0,\cdot)   + \int_0^T \int_\B \lambda \theta^4 \vphi \notag \\
& 	\geq 
\int_0^T\int_{\B} \frac{\vphi}{\theta}  \left( \Ss_\omega : \nabla_x \uvect  + \frac{\kappa_\nu(\theta)|\nabla_x\theta|^2  }{\theta} - \sum_{k=1}^n \rho (h_k - \theta s_k)\sigma_k \right)  
	\end{align}
for any  test function $\vphi \in \C^1([0,T)\times \B; \mathbb R)$ with $\vphi \geq 0$.

\subsection{Energy balance} Let us write the  energy equation   for the penalized system, which is   
\begin{align}\label{energy-pena-1}
&	-	\int_0^T \partial_t \psi \int_{\B}  \left(\frac{1}{2} \rho |\uvect|^2  + \rho e_M(\rho,\theta)  + \rho e_{R,\xi}(\rho, \theta)  + \frac{\delta}{\beta-1} \rho^\beta \right) \notag \\
	&	 + \int_0^T \int_\B \lambda \theta^5 \psi	
		+
		 \frac{1}{\veps} \int_0^T \int_{\Gamma_t} \left|(\uvect -\V) \cdot \mathbf n\right|^2 \psi  \notag \\
	 =&	-\int_0^T  \psi \int_{\B}  \big(\rho[\uvect \otimes \uvect]    : \nabla_x \V  - \Ss_\omega : \nabla_x \V  + p_{\xi, \delta}(\rho,\theta) \Div_x \V  \big) 
		+ \int_0^T \int_{\B} \partial_t(\rho \uvect) \cdot \V  \psi  
		\notag \\ 
 &  -\int_0^T \psi \int_{\B} \sum_{k=1}^n \rho h_k \sigma_k +\psi(0) \int_{\B} 	  \left(\frac{1}{2} \frac{|(\rho \uvect)_{0,\delta}|^2}{\rho_{0,\delta}}  + \rho_{0,\delta} \left(e_M + e_{R,\xi}\right)(\rho_{0,\delta},\theta_{0,\delta})  + \frac{\delta}{\beta-1} \rho_{0,\delta}^\beta \right)
\end{align}
for any $\psi \in \C^1_c([0,T))$ such that $\psi\geq 0$ and $\partial_t \psi\leq 0$.

%
%

\begin{definition}
\label{weak-solution-penalization}
We say that the quantity
$(\rho,\uvect,\theta, \Yvect)$ is a weak solution to the penalized problem with the modified constitutive relations in Section \ref{Section-hypo-molified}, initial data \eqref{pena-rho-initial}--\eqref{pena-Y-initial} and boundary data \eqref{boundary-extended-1}--\eqref{boundary-extended-2} if the following items hold:
\begin{itemize}
			\item 
			$\rho\in L^\infty(0,T;L^{\frac{5}{3}}(\mathbb{R}^3))$,
		 $\rho\geq 0$, $\rho\in L^q(Q_T)$ with certain $q>1$,
		
		\item $\uvect, \, \nabla_x\uvect \in L^2((0,T)\times \B)$,
		$\rho \uvect \in L^\infty(0,T; L^1(\mathbb R^3))$,
		
		\item $\theta>0$  a.e. on $(0,T)\times \B$, $\theta \in L^\infty(0,T;L^4(\mathbb{R}^3))$,
		  $\theta, \nabla_x\theta, \log\theta, \nabla_x\log\theta \in L^2((0,T)\times \B)$,

  \item   $\rho s, \, \rho s\uvect  \in L^1( (0,T)\times \B )$,

  \item $Y_k \in L^\infty( (0,T)\times \B ), \, Y_k \in L^2( 0,T; W^{1,2}(\B))$
   for each $1\leq k \leq n$, and

  \item the relations \eqref{weak-conti} to \eqref{energy-pena-1} hold. 
  
\end{itemize}
    
\end{definition}

\begin{theorem}\label{weak-solution-fixed-domain}
Assume that $\V\in \C^1([0,T]; \C_c^3(\mathbb{R}^3;\mathbb{R}^3 ))$   the hypotheses in Section \ref{Section-hypo-molified} are satisfied.  
Moreover, the initial data satisfy \eqref{pena-rho-initial}--\eqref{pena-Y-initial} and boundary coniditions are as considered in \eqref{boundary-extended-1}--\eqref{boundary-extended-2}. 
Then there exists a weak solution to the penalized problem on any time interval $(0,T)$ in the sense of Definition \ref{weak-solution-penalization}.  
\end{theorem}

\begin{proof}
 The existence of weak solution of the concerned penalized system 
in the fixed domain is similar to \cite{Eduard-CPAA-2008}. We just give a short sketch  of the proof. It is necessary to  regularize the continuity equation with a viscous term $\Delta_x \rho$ and add appropriate terms in the momentum,  energy equations and the equations of species mass fractions as described in \cite[Section 5]{Eduard-CPAA-2008}. We need to  solve the momentum equations via Faedo-Galerkin approximations and then the energy equation and  equations for species mass fractions. 


However, as pointed out in \cite[Theorem 3.1]{Sarka-et-al-ZAMP},  we also face the following difficulties here. 
\begin{itemize}
\item The penalized terms  $\frac{1}{\veps} \int_0^T \int_{\Gamma_t}  (\uvect -\V) \cdot \mathbf n \ \boldvphi \cdot \mathbf n$ 
in \eqref{weak-momen} and $\frac{1}{\veps} \int_0^T \psi \int_{\Gamma_t} |(\uvect - \V)\cdot \mathbf n|^2$   in   \eqref{energy-pena-1}. 
	
\item The jumps in functions $\kappa_\nu (\theta, t, x)$ in \eqref{defi-kappa-nu} and $a_\xi( t, x)$ in \eqref{a-xi}. 
\end{itemize}

The strategy to overcome these difficulties  has already been discussed at the beginning of the proof of Theorem 3.1 in \cite{Sarka-et-al-ZAMP}. In the present work, we employ a
  similar methodology for the proof.  We emphasize that the term $\lambda\theta^5$  is necessary in the  internal energy equation 
  \eqref{energy-pena-1} to provide uniform bounds of high power of the temperature in $(0,T)\times \B$. 

  A detailed study of the existence of such  approximated solutions  can be
  found in \cite{KMNPW-L} in the context of compressible fluids.  
\end{proof}

The rest of the paper is devoted to prove the main result of this work, that is, Theorem \ref{Main-theorem}.

\section{Modified energy inequality and uniform bounds} 

\subsection{Finding a modified energy inequality}

Let us consider
$\psi_\varsigma \in \C^1_c([0,T))$ such that 
\begin{align}\label{function-psi-zeta}
	\psi_\varsigma(t) = 
	\begin{dcases} 1 \quad &\text{for } t < \tau - \varsigma , \\
		0   \quad & \text{for } t \geq \tau ,     
		\end{dcases}  \ \ \text{ for  given } \tau \in (0,T) \text{ and } \ 0< \varsigma < \tau   .
	\end{align}
We now use the test function $\vphi(t,x) = 1 \cdot \psi_\varsigma(t)$ in the entropy inequality \eqref{penalized-entropy-1} so that we have (after passing to the limit $\varsigma \to 0$)
 \begin{align}\label{entropy-particular-aux}
 	-	\int_\B \rho s_\xi
  (\tau, \cdot) 
 	+ \int_0^\tau \int_{\B} \frac{1}{\theta}  \left( \Ss_\omega : \nabla_x \uvect  + \frac{\kappa_\nu(\theta)|\nabla_x\theta|^2  }{\theta} - \sum_{k=1}^n \rho (h_k - \theta s_k) \sigma_k \right)  - \int_0^\tau \int_\B \lambda \theta^4 \notag \\
 	 \leq
 	- \int_{\B} \rho_{0,\delta} s_{\xi}(\rho_{0,\delta}, \theta_{0,\delta}, \Yvect_{0,\delta} )  
 \end{align} 
for a.a.  $\tau \in [0,T]$.

\vspace*{.1cm}

Upon using  the same test function  
 $\vphi(t,x) = 1 \cdot \psi_\varsigma(t)$  in the formulation \eqref{eq-mass-species-1-pena}, we have (after passing to the limit $\varsigma \to 0$), 
\begin{equation*}
\begin{aligned}
\int_{\B} \rho Y_k(\tau, \cdot) = \int_0^\tau \int_\B \rho \sigma_k + \int_\B \rho_{0,\delta} Y_{k,0,\delta} 
\end{aligned}
\end{equation*} 
for a.a.  $\tau \in [0,T]$. In particular, 
\begin{align}\label{eq-mass-species-aux}
\int_{\B}\sum_{k=1}^n \rho s_k  Y_k(\tau, \cdot) = \int_0^\tau \int_\B \sum_{k=1}^n \rho s_k \sigma_k + \int_\B \sum_{k=1}^n \rho_{0,\delta} s_k  Y_{k,0,\delta} 
\end{align} 
for a.a. $\tau \in [0,T]$, since $s_k$ for $k=1, ..., n$ are constants.

Now, recall that 
\begin{align}\label{exp-new-s}
s_\xi(\rho, \theta, \Yvect) = s_M(\rho, \theta)  + s_{R,\xi}(\rho, \theta)    + \sum_{k=1}^n s_k Y_k  \quad \text{(with   $s_{R,\xi}(\rho, \theta) = \frac{4a_\xi \theta^3}{3\rho}$)} .
\end{align}

Using this in \eqref{entropy-particular-aux} and thanks to  \eqref{eq-mass-species-aux}, we get 
\begin{align}\label{entropy-particular}
-\int_\B \left[\rho (s_M+s_{R,\xi})\right](\tau, \cdot)  + \int_0^\tau \int_{\B} \frac{1}{\theta}  \left( \Ss_\omega : \nabla_x \uvect  + \frac{\kappa_\nu(\theta)|\nabla_x\theta|^2  }{\theta} - \sum_{k=1}^n \rho h_k \sigma_k \right)  - \int_0^\tau \int_\B \lambda \theta^4  \notag \\
 	 \leq
 	- \int_{\B} \rho_{0,\delta} (s_M+s_{R,\xi})(\rho_{0,\delta}, \theta_{0,\delta})   
\end{align}
for a.a.  $\tau \in [0,T]$.

In the next step, we use $\psi=\psi_\varsigma$ as a test function in \eqref{energy-pena-1} and  we derive after passing to the limit $\varsigma \to 0$, that
\begin{align}\label{energy-pena-2}
		& \int_{\B}  \left(\frac{1}{2} \rho |\uvect|^2  + \rho e_M(\rho,\theta)  + \rho e_{R,\xi}(\rho, \theta)  + \frac{\delta}{\beta-1} \rho^\beta \right)(\tau, \cdot)  \notag \\
		&	\quad + \int_0^\tau \int_\B \lambda \theta^5
		+
		\frac{1}{\veps} \int_0^\tau \int_{\Gamma_t} \left|(\uvect -\V) \cdot \mathbf n\right|^2  \notag  \\
		=&	-\int_0^\tau  \int_{\B}  \big(\rho[\uvect \otimes \uvect]    : \nabla_x \V  - \Ss_\omega : \nabla_x \V  + p_{\xi, \delta}(\rho,\theta) \Div_x \V  + \rho \uvect \cdot \partial_t \V   \big) \notag \\
		& + \int_\B  (\rho \uvect \cdot \V)(\tau, \cdot)  -\int_\B   (\rho \uvect)_{0,\delta} \cdot \V(0, \cdot)  - \int_0^\tau \int_{\B} \sum_{k=1}^n \rho h_k \sigma_k 
		\notag \\ 
		& + \int_{\B} 	  \left(\frac{1}{2} \frac{|(\rho \uvect)_{0,\delta}|^2}{\rho_{0,\delta}}  + \rho_{0,\delta} (e_M+e_{R, \xi})(\rho_{0,\delta},\theta_{0,\delta})  + \frac{\delta}{\beta-1} \rho_{0,\delta}^\beta    \right)
	\end{align}
for a.a. $\tau \in [0,T].$

Now, by summing up  \eqref{entropy-particular} and \eqref{energy-pena-2}, we deduce that 
 \begin{align}\label{energy-pena-3}
 		& \int_{\B}  \left(\frac{1}{2} \rho |\uvect|^2  + \rho( e_M +  e_{R,\xi}) - \rho (s_M +s_{R, \xi})  + \frac{\delta \rho^\beta}{\beta-1}\right)(\tau, \cdot) 
+  \int_0^\tau \int_\B \lambda \theta^5 \notag \\
 	&	+
 		\frac{1}{\veps} \int_0^\tau \int_{\Gamma_t} \left|(\uvect -\V) \cdot \mathbf n\right|^2  + \int_0^\tau \int_\B \frac{1}{\theta}    \left( \Ss_\omega : \nabla_x \uvect  + \frac{\kappa_\nu(\theta)|\nabla_x\theta|^2  }{\theta} - \sum_{k=1}^n \rho h_k \sigma_k \right) \notag   \\
 	\leq &	-\int_0^\tau  \int_{\B}  \big(\rho[\uvect \otimes \uvect]    : \nabla_x \V  - \Ss_\omega : \nabla_x \V  + p_{\xi, \delta}(\rho,\theta)\, \Div_x \V  + \rho \uvect \cdot \partial_t \V   \big) \notag \\
 		& + \int_\B  (\rho \uvect \cdot \V)(\tau, \cdot) - \int_\B (\rho \uvect)_{0,\delta} \cdot \V(0, \cdot)    - \int_0^\tau \int_{\B} \sum_{k=1}^n \rho h_k \sigma_k  + \int_0^\tau \int_\B\lambda \theta^4     \notag \\ 
 		& + \int_{\B} 	  \left(\frac{1}{2} \frac{|(\rho \uvect)_{0,\delta}|^2}{\rho_{0,\delta}}  + \rho_{0,\delta} (e_M+ e_{R, \xi})(\rho_{0,\delta},\theta_{0,\delta}) - \rho_{0,\delta} (s_M+ s_{R, \xi})(\rho_{0,\delta},\theta_{0,\delta})   + \frac{\delta \rho_{0,\delta}^\beta}{\beta-1}     \right)
 	\end{align}
 for a.a. $\tau \in [0,T].$ 
 
\vspace*{.1cm}

Inspired by \cite[Chapter 2.2.3]{Feireisl-Novotny-book},  we hereby   define the Helmholtz function
\begin{align}\label{H-B}
	\mathcal H_{1, \xi}(\rho, \theta) : =
		\rho \left(e_M(\rho, \theta) + e_{R, \xi}(\rho, \theta) \right)  - \rho \left(s_{M}(\rho, \theta) + s_{R, \xi}(\rho, \theta) \right),
	\end{align}
	and in what follows, we denote  
	\begin{align}\label{H-B-0}
		\mathcal	H_{1, \xi}(\rho_{0,\delta}, \theta_{0,\delta}) &= 	\rho_{0,\delta} \left(e_M(\rho_{0,\delta}, \theta_{0,\delta}) + e_{R, \xi}(\rho_{0,\delta}, \theta_{0,\delta}) \right) \notag \\
		& \quad   - \rho_{0,\delta} \left(s_{M}(\rho_{0,\delta}, \theta_{0,\delta}) + s_{R, \xi}(\rho_{0,\delta}, \theta_{0,\delta})\right) . 
	\end{align}

\subsection{Finding uniform bounds}

Let us now find suitable bounds for  the right-hand sides of the modified energy inequality \eqref{energy-pena-3}. 
First, we recall that the fluid system satisfies the {\em total mass conservation law}, that is, 
\begin{align*}
	\int_{\B} \rho(\tau, \cdot ) = \int_\B \rho_{0,\delta} (\cdot) = \int_{\Omega_0} \rho_0(\cdot) =  C(\rho_0)>0 .
\end{align*}

\vspace*{.1cm}
\noindent 
{$\bullet$ \bf Step 1.}  
(i)  For any $\epsilon>0$ small, we have 
\begin{align}\label{esti-1}
	\int_\B  (\rho \uvect \cdot \V)(\tau,\cdot)  \leq C(\V) 	\left| \int_\B  \sqrt{\rho} \sqrt{\rho} \uvect(\tau,\cdot) \right| \leq C(\V, \rho_0)  + \epsilon \int_\B \rho |\uvect|^2 .
\end{align} 

\vspace*{.2cm}
\noindent 
(ii) 
Without loss of generality, we assume $0<\lambda\leq 1$ from now onward.  Then, by using H\"older's and Cauchy-Schwarz inequalities, we obtain 
\begin{align}\label{esti-2} 
	\int_0^\tau \int_\B \Ss_\omega : \nabla_x \V 
	&\leq   \frac{1}{2} \int_0^\tau \int_\B  \frac{1}{\theta} \Ss_\omega : \nabla_x \uvect + C(\V) \int_0^\tau \int_\B \theta \notag  \\
	& \leq \frac{1}{2} \int_0^\tau \int_\B  \frac{1}{\theta} \Ss_\omega : \nabla_x \uvect + 
	\epsilon \int_0^\tau \int_\B \lambda \theta^5 +  \frac{C(\V, \epsilon)}{\lambda^{\frac{1}{4}}}. 
\end{align}
We also have that 
\begin{align}\label{esti-3}
	\left| \int_0^\tau \int_\B\rho[\uvect \otimes \uvect]  : \nabla_x \V  \right| \leq  C(\V)\int_0^\tau \int_\B \rho |\uvect|^2   \ \text{ and}
\end{align}
\begin{align}\label{esti-4}
	\left| \int_0^\tau \int_\B\rho \uvect \cdot \partial_t \V  \right| \leq  C(\V, \rho_0) + C \int_0^\tau \int_\B \rho |\uvect|^2  .
\end{align}

\vspace*{.2cm}
\noindent 
(iii) 
Next, since $0<\lambda\leq 1$,  it is easy to observe that
\begin{align}\label{esti-5}
	\int_0^\tau \int_\B \lambda \theta^4  \leq C(\epsilon) + \epsilon \int_0^\tau \int_\B \lambda \theta^5 .
\end{align}

\vspace*{.2cm}
\noindent 
(iv) 
In order to deal with  pressure term  $p_{\xi, \delta}(\rho, \theta)$, let us notice that 
\begin{align}\label{fact-1}
	P^\prime(Z)>0,  \quad \forall Z>0.
\end{align}
Indeed, \eqref{hypo-p_m} gives that $P^\prime(Z)>0$ if  $0<Z<\underline Z$, or, $Z>\overline Z$. This together with \eqref{hypo_P} gives \eqref{fact-1} provided  we extend $P$ as a strictly increasing function in $[\underline Z, \overline Z]$. 
Next, we  recall the fact \eqref{fact-2-1} which tells that
$\disp \lim_{Z\to \infty} \frac{P(Z)}{Z^{\frac{5}{3}}} = p_\infty > 0$. Using this, the fact \eqref{fact-1} and the information \eqref{hypo_P}--\eqref{hypo-p_m_with_P-2}, we obtain the following bounds on the molecular pressure $p_M$, 
\begin{equation}\label{mole_bound_p_M}
	\begin{aligned} 
		\underline c \rho \theta \leq & p_M \leq \overline c \rho \theta      \quad  \text{if } \rho < \overline Z \theta^{\frac{3}{2}}, \\
			\underline c \rho^{\frac{5}{3}} \leq & p_M \leq 
			\begin{cases}
				\overline c \theta^{\frac{5}{2}}  \quad  \text{if } \rho < \overline Z \theta^{\frac{3}{2}} , \\
				\overline c	\rho^{\frac{5}{3}} \quad  \text{if } \rho > \overline Z \theta^{\frac{3}{2}} ,
			\end{cases}
	\end{aligned}
\end{equation}
and $p_M$ is monotone in $\underline Z \theta^{\frac{3}{2}} \leq \rho \leq \overline Z \theta^{\frac{3}{2}}$.  

With the above information, we deduce that 
\begin{align} \label{estimate-pressue-term}
\left|\int_0^\tau \int_\B p_{\xi, \delta} (\rho, \theta) \Div_x \V \right| 
 \leq C(\V) \int_0^\tau \int_\B \frac{\delta}{\beta-1} \rho^\beta +  C(\V) \int_0^\tau \int_\B a_{\xi}  \theta^4  \notag \\ 
 + C(\V) \int_0^\tau \int_\B \rho^{\frac{5}{3}} 
 + \epsilon \int_0^\tau \int_\B \lambda \theta^5 + \frac{C(\V, \epsilon)}{\lambda} .
\end{align}

Let us observe that 
\begin{align}\label{lower_bound_rho_e_xi}
\rho (e_M(\rho, \theta) + e_{R,\xi}(\rho ,\theta)  ) \geq a_\xi \theta^4 + \frac{3}{2} p_\infty \rho^{\frac{5}{3}}	,
\end{align}
which can be shown in the same way as we have obtained 
\eqref{lower_bound_rho_e}, and 
 therefore, 
\begin{align*}
	\int_0^\tau \int_\B \left(a_\xi \theta^4 + \rho^{\frac{5}{3}} \right)  \leq C(p_\infty) \int_0^\tau \int_\B \rho (e_M + e_{R,\xi}) .
\end{align*}

\vspace*{.1cm}

Using the above fact in \eqref{estimate-pressue-term},  together with all other estimates above, we have from \eqref{energy-pena-3}  (by fixing $\epsilon>0$ small enough) the following:
\begin{align}\label{energy-pena-4}
& \int_{\B}  \left(\frac{1}{2} \rho |\uvect|^2  + \mathcal H_{1, \xi}(\rho, \theta)  + \frac{\delta \rho^\beta}{\beta-1} \right)(\tau, \cdot) 
+  \int_0^\tau \int_\B \lambda \theta^5  \notag \\
&	+
\frac{1}{\veps} \int_0^\tau \int_{\Gamma_t} \left|(\uvect -\V) \cdot \mathbf n\right|^2  + \int_0^\tau \int_\B \frac{1}{\theta}    \left( \Ss_\omega : \nabla_x \uvect  + \frac{\kappa_\nu(\theta)|\nabla_x\theta|^2  }{\theta} - \sum_{k=1}^n \rho h_k \sigma_k \right)
\notag \\
\leq  &
\int_{\B} \left(\frac{1}{2}  \frac{|(\rho\uvect)_{0,\delta}|^2}{\rho_{0,\delta}} + \mathcal  	H_{1, \xi}(\rho_{0,\delta}, \theta_{0,\delta})  + \frac{\delta}{\beta-1} \rho^\beta_{0,\delta}  - (\rho \uvect)_{0,\delta} \V(0,\cdot) \right) 
	\notag 		\\
	&	+ C\int_0^\tau \int_\B \left( \frac{1}{2}\rho |\uvect|^2   +  \rho (e_M + e_{R,\xi}) 
	 +   \frac{\delta}{\beta-1} \rho^\beta \right) \notag \\
&	 + \int_{0}^\tau \int_\B \sum_{k=1}^n \left|\rho h_k \sigma_k \right| + C\left(1 + \frac{1}{\lambda}\right)   
		\end{align}
for a.a.  $\tau \in [0,T]$, where $\mathcal H_{1, \xi}(\rho, \theta)$ and $\mathcal H_{1, \xi}(\rho_{0,\delta}, \theta_{0,\delta})$ are defined by \eqref{H-B} and \eqref{H-B-0} respectively, and the constant $C>0$  above  depends on the quantities  $\V, \rho_0$, $p_\infty$ but not on the parameters $\lambda$, $\omega$, $\xi$, $\nu$, $\veps$ or $\delta$.   
In the above, we could replace the quantity $\disp C\left(1+\frac{1}{\lambda}\right)$ by $\disp \frac{C}{\lambda}$ 
since we already assumed that $0<\lambda \leq 1$.

\vspace*{.15cm} 
\noindent 
$\bullet$ {\bf Step 2.} 
(i) It is straightforward to  deduce that 
\begin{align}\label{rho-h_k-sigma_k} 
	\int_0^\tau \int_\B \sum_{k=1}^n \left|\rho h_k \sigma_k \right| \leq C 
\end{align}
for some constant $C>0$ that depends on $\rho_0$, $h_k$ and $\overline \sigma$, 
thanks to the properties of species production rate $\sigma_k$ given by \eqref{coeff-sigma-k} and the fact that $h_k$ are constants for $k=1,...,n$. 

\vspace*{.2cm}
\noindent
(ii) Next, on the left-hand side of \eqref{energy-pena-3}, we have 
\begin{align}\label{rho-h-sig-thet} 
- \int_0^\tau \int_\B \sum_{k=1}^n \frac{1}{\theta} \rho h_k \sigma_k & = - \int_0^\tau \int_\B \sum_{k=1}^n \frac{1}{\theta} \rho  \sigma_k (g_k + \theta s_k)  \quad (\text{using } \eqref{retaion-g_k}) \notag \\
	& =  - \int_0^\tau \int_\B \sum_{k=1}^n \frac{1}{\theta} \rho  g_k \sigma_k       - \int_0^\tau \int_\B \sum_{k=1}^n \rho \sigma_k s_k  \notag \\
	& \geq -  \sum_{k=1}^n \overline \sigma s_k \int_0^\tau \int_\B \rho   \quad (\text{using } 
	\eqref{sum-sigma-k}) .
\end{align}

\vspace*{.1cm}

Using  \eqref{rho-h_k-sigma_k} and   \eqref{rho-h-sig-thet}, 
the inequality  \eqref{energy-pena-4} follows to
\begin{align}\label{energy-pena-4-2}
	& \int_{\B}  \bigg(\frac{1}{2} \rho |\uvect|^2  + \mathcal H_{1, \xi}(\rho, \theta)  + \frac{\delta \rho^\beta}{\beta-1}\bigg)(\tau, \cdot) 
+  \int_0^\tau \int_\B \lambda \theta^5 +
\frac{1}{\veps} \int_0^\tau \int_{\Gamma_t} \left|(\uvect -\V) \cdot \mathbf n\right|^2  \notag \\
& \ + \int_0^\tau \int_\B  \frac{\kappa_\nu(\theta)|\nabla_x\theta|^2  }{\theta^2} + \int_0^\tau \int_{\B} \frac{1}{\theta} \Ss_\omega : \nabla_x \uvect  
			\notag \\
	\leq  &
	\int_{\B} \left(\frac{1}{2}  \frac{|(\rho\uvect)_{0,\delta}|^2}{\rho_{0,\delta}} + \mathcal  	H_{1, \xi}(\rho_{0,\delta}, \theta_{0,\delta})  + \frac{\delta}{\beta-1} \rho^\beta_{0,\delta}  - (\rho \uvect)_{0,\delta} \V(0,\cdot) \right) 
			\notag \\
	&	+ C\int_0^\tau \int_\B \left( \frac{1}{2}\rho |\uvect|^2   +  \rho (e_M + e_{R,\xi}) 
	 +   \frac{\delta \rho^\beta }{\beta-1} \right) + \frac{C}{\lambda} 
		\end{align}
for a.a. $\tau \in [0,T]$, where the constant $C>0$ is independent of the parameters  $\lambda$, $\omega$, $\xi$, $\nu$, $\veps$ and $\delta$.

\vspace*{.1cm}

Then, by means of    
 Gr\"{o}nwall's inequality we  get
\begin{align} \label{energy-pena-5}
	& \int_{\B}  \left(\frac{1}{2} \rho |\uvect|^2  + \mathcal H_{1, \xi}(\rho, \theta)  + \frac{\delta \rho^\beta}{\beta-1}\right)(\tau, \cdot) 
+  \int_0^\tau \int_\B \lambda \theta^5 +
\frac{1}{\veps} \int_0^\tau \int_{\Gamma_t} \left|(\uvect -\V) \cdot \mathbf n\right|^2  \notag \\
& \ + \int_0^\tau \int_\B  \frac{\kappa_\nu(\theta)|\nabla_x\theta|^2  }{\theta^2} + \int_0^\tau \int_{\B} \frac{1}{\theta} \Ss_\omega : \nabla_x \uvect 
\notag  \\
	\leq  &
	\int_{\B} \left(\frac{1}{2}  \frac{|(\rho\uvect)_{0,\delta}|^2}{\rho_{0,\delta}} + \mathcal  	H_{1, \xi}(\rho_{0,\delta}, \theta_{0,\delta})  + \frac{\delta}{\beta-1} \rho^\beta_{0,\delta}  - (\rho \uvect)_{0,\delta} \V(0,\cdot) \right) +\frac{C}{\lambda} 
 \end{align}
for a.a. $\tau \in [0,T]$. 

\vspace*{.1cm}

To ensure that the left-hand side is positive, we shall do the following. Setting a constant $\overline \rho$  such that $\disp \int_\B (\rho-\overline \rho)=0$  for almost all $\tau \in [0,T]$ and, we rewrite the estimate \eqref{energy-pena-5} as follows
 \begin{align}\label{energy-pena-6}
 		&	\int_{\B} \left(\frac{1}{2} \rho |\uvect |^2 +	\mathcal H_{1, \xi}(\rho, \theta) - 
 		(\rho-\overline \rho) \frac{\partial\mathcal  H_{1, \xi}(\overline \rho, 1) }{\partial \rho} - \mathcal H_{1, \xi}(\overline \rho, 1) + \frac{\delta \rho^\beta}{\beta-1} \right)(\tau, \cdot)   \notag \\
 		&	+ \frac{1}{\veps} \int_0^\tau \int_{\Gamma_t} \left|(\uvect -\V) \cdot \mathbf n \right|^2 
 		+	\int_0^\tau \int_\B \lambda \theta^5	\notag \\  
 		& + \int_0^\tau \int_{\B}   \frac{\kappa_\nu(\theta) |\nabla_x \theta|^2}{\theta^2}   + 
 		\int_0^\tau \int_{\B} \frac{1}{\theta} \Ss_\omega : \nabla_x \uvect 
 	\notag 	\\
 		\leq 
 		&
 		C	\int_{\B} \left(\frac{1}{2}  \frac{|(\rho\uvect)_{0,\delta}|^2}{\rho_{0,\delta}}  +  \mathcal H_{1, \xi}(\rho_{0,\delta}, \theta_{0,\delta}) + \frac{\delta}{\beta-1} \rho^\beta_{0,\delta}  - (\rho \uvect)_{0,\delta} \V(0,\cdot)  \right)  \notag \\
 	& \ - \int_\B \left((\rho_{0,\delta}-\overline \rho) \frac{\partial \mathcal H_{1, \xi}(\overline \rho, 1) }{\partial \rho} + \mathcal H_{1, \xi}(\overline \rho, 1)\right) 
  + \frac{C}{\lambda} 
 	\end{align}
for a.a. $\tau \in [0,T]$, where the constant $C>0$ does not depend on any of the parameters.

\vspace*{.2cm}
\noindent 
$\bullet$ {\bf The uniform bounds.} (i) From \eqref{energy-pena-6}, we directly have 
\begin{align}\label{uniform-bound-1}
	\int_0^T \int_{\Gamma_t} \left|(\uvect - \V)\cdot \mathbf n\right|^2 & \leq \frac{\veps \, C}{\lambda}  
,  \\
	  \esssup_{\tau \in [0,T]} \left\|\delta \rho^\beta(\tau, \cdot)\right\|_{L^1(\B)}    &\leq \frac{C}{\lambda} ,             \label{uniform-bound-2} \\	  
	    \esssup_{\tau \in [0,T]} \left\| \sqrt{\rho} \uvect(\tau, \cdot) \right\|^2_{L^2(\B; \mathbb R^3)}    &\leq \frac{C}{\lambda} ,            \label{uniform-bound-3} \\		 
	\text{and } \    \left\| \lambda \theta^5 \right\|_{L^1((0,T)\times \B)} &\leq \frac{C}{\lambda}  \label{uniform-bound-4}.
\end{align}

\vspace*{.2cm} 
\noindent
(ii) We further have 
\begin{align}\label{uniformbound-Stress}
\int_0^T \int_\B \frac{1}{\theta} \Ss_\omega : \nabla_x \uvect \leq  \frac{C}{\lambda}  .
\end{align} 

Let us
 recall the expression of $\Ss_\omega$ from \eqref{stress_tensor-omega}. With this in hand,  and by using   \eqref{def-mu-omega}, \eqref{def-nu-omega} and \eqref{hypo-mu},
we obtain 
\begin{align}\label{lower-bound-S-w-u}
	\int_0^\tau \int_\B \frac{1}{\theta} \Ss_\omega : \nabla_x \uvect \geq c_1 (\omega) 
	 \int_0^\tau \int_\B \left|\nabla_x \uvect + \nabla^\top_x \uvect -\frac{2}{3}\Div_x \uvect \mathbb I   \right|^2 
\end{align}
for some constant $c_1(\omega)>0$. 

On the other hand, by the Korn-Poincar\'e inequality (see Lemma 
\ref{Korn-Poincare}), we have 
\begin{equation*}
	\begin{aligned}
		\|\uvect\|^2_{W^{1,2}_0(\B; \mathbb R^3)} &\leq C \left\|\nabla_x \uvect + \nabla^\top_x \uvect -\frac{2}{3}\Div_x \uvect \mathbb I   \right\|^2_{L^2(\B; \mathbb R^3)} + C \left(\int_\B \rho |\uvect| \right)^2  \\
		& \leq  C \left\|\nabla_x \uvect + \nabla^\top_x \uvect -\frac{2}{3}\Div_x \uvect \mathbb I   \right\|^2_{L^2(\B; \mathbb R^3)} + C(\rho_0)\int_\B \rho |\uvect|^2 .
	\end{aligned} 
\end{equation*}
Therefore, 
\begin{align}\label{u_W_12}
	c_1(\omega) \int_0^\tau	\|\uvect\|^2_{W^{1,2}_0(\B; \mathbb R^3)}  \leq C  \int_0^\tau \int_\B \frac{1}{\theta} \Ss_\omega : \nabla_x \uvect + C(\rho_0) \int_0^\tau \int_\B \rho |\uvect|^2, 
\end{align}
and consequently, by \eqref{uniform-bound-3} and \eqref{uniformbound-Stress}, we have 
\begin{align}\label{uniform-bound-5}
c_1(\omega) \|\uvect\|^2_{L^2(0,T; W^{1,2}_0(\B; \mathbb R^3))} \leq   \frac{C}{\lambda}.  
\end{align} 
\begin{remark}\label{remark-c_1-c_2-omega}
	From the definitions of $\mu_\omega$ and $\eta_\omega$ in \eqref{def-mu-omega}--\eqref{def-nu-omega}, it is clear that the constant $c_1(\omega)$ 
	behaves like $c \, \omega$ in $((0,T)\times \B)\setminus Q_T$    for some constant $c>0$ that is independent of $\omega$. 
\end{remark}

\vspace*{.1cm}
\noindent
(iii) We further have (from \eqref{energy-pena-6}), 
\begin{align}\label{esti-temp-sub}
\int_0^\tau \int_\B \frac{\kappa_\nu(\theta) |\nabla_x \theta|^2}{\theta^2} \leq \frac{C}{\lambda}.
\end{align}
    Recall that $\kappa_\nu(\theta) = \chi_\nu \kappa(\theta)$ (see \eqref{defi-kappa-nu}) and $\kappa(\theta)=\kappa_M(\theta)+\kappa_R(\theta)$
     which are defined by 
     \eqref{hypo-kappa}; this yields 
\begin{align}\label{bound-theta-derivative}
	\int_0^T \int_\B \left( |\nabla_x \log(\theta)|^2 + |\nabla_x \theta^{\frac{3}{2}}|^2  \right) \leq \frac{C(\nu)}{\lambda}.   
\end{align}

\vspace*{.2cm}
\noindent
(iv)
Now, since $\mathcal H_{1, \xi}$ is coercive (see, for instance, \cite[Proposition 3.2]{Feireisl-Novotny-book}) and bounded from below, we get 
\begin{align}\label{bound-rho-e-xi}
\esssup_{\tau\in [0,T]} \|\rho(e_M+e_{R,\xi})(\tau, \cdot)\|_{L^1(\B)} \leq \frac{C}{\lambda} ,
\end{align}
and consequently we have 
\begin{align}\label{bound-theta-L4} 
	 \esssup_{\tau\in (0,T)} \|a_\xi \theta^4(\tau, \cdot)\|_{L^1(\B)} \leq \frac{C}{\lambda}, \\
	 \label{rho_5/3}  \esssup_{\tau\in (0,T)} \|\rho(\tau, \cdot)\|^{\frac{5}{3}}_{L^{\frac{5}{3}}(\B)}  \leq     \frac{C}{\lambda} .
\end{align}

\vspace*{.2cm}
\noindent
(v) 
Then by \eqref{bound-theta-derivative}, \eqref{bound-theta-L4} and the generalized  Poincar\'{e} inequality from \Cref{Poincare} (since the condition 
\eqref{measurable-condition} satisfies), we deduce that 
\begin{align}\label{bound-theta-gamma}
	&\|\theta^\gamma \|_{L^2(0,T; W^{1,2}(\B))} \leq C(\lambda, \nu) \quad \text{for any } 1 \leq \gamma \leq \frac{3}{2},  \text{ and} \\
\label{bound-log-theta}
&\|\nabla_x \log \theta \|_{L^2(0,T; L^2(\B;\mathbb R^3))} \leq  C(\lambda, \nu) .
\end{align}

\vspace*{.2cm} 
 \noindent 
 (vi)
By using  \eqref{esti-temp-sub} and 
\eqref{bound-theta-gamma},      we get
\begin{align}\label{bound-kappa-theta-2}
\int_0^T \int_\B	\frac{\kappa_\nu(\theta)}{\theta} |\nabla_x \theta| 
&
\leq \frac{1}{2}	\int_0^T \int_\B \frac{\kappa_\nu(\theta)}{\theta^2} |\nabla_x \theta|^2 + \frac{1}{2} \int_0^T \int_\B \kappa_\nu(\theta) \leq C(\lambda, \nu).
\end{align} 



%

\vspace*{.2cm}
\noindent 
(vii) By virtue of  \eqref{uniform-bound-3} and \eqref{rho_5/3}, we further deduce that 
\begin{align}\label{bound-rho-u-5/4}
\|\rho \uvect \|_{L^\infty(0,T; L^{\frac{5}{4}}(\Omega; \mathbb R^3) } \leq C(\lambda) .
\end{align} 
Moreover, the convective term satisfies
\begin{align*}
\bigg| \int_{\B} \rho [\uvect \otimes \uvect] \cdot  \boldphi \bigg|
\leq \|\rho \uvect\|_{L^{\frac{5}{4}}(\B; \R^3) } \|\uvect\|_{L^6(\B;\R^3)} \|\boldphi\|_{L^{30}(\B; \R^3)},
\end{align*} 
for any $\boldphi \in L^2(0,T; L^{30}(\B; \R^3))$,
and consequently, 
\begin{align}\label{bound-convective-term} 
\|\rho [\uvect \otimes \uvect] \|_{L^2(0,T;L^{\frac{30}{29}}(\B; \R^3))     } \leq C(\lambda, \omega) ,
\end{align} 
thanks to \eqref{uniform-bound-5} and \eqref{bound-rho-u-5/4}.

\vspace*{.2cm}
\noindent 
(viii) {\bf Pressure estimate.}
Using the technique based on the {\em Bogovskii operator} (see, for instance  \cite{Feireisl-hana})
one can get more information about the modified pressure 
$p_{\xi, \delta}(\rho, \theta)$, namely,  
\begin{align}\label{Bogovskii}
\iint_K   p_{\xi,\delta} (p,\theta) \rho \leq C(\lambda, \xi, \nu, \delta)
\end{align}
for any compact set $K\subset (0,T)\times \B$ 
 such that 
\begin{align}\label{Set-K}
K \cap \left(\cup_{\tau\in [0,T]} \big(\{\tau\} \times \Gamma_\tau\big)  \right) = \emptyset .
\end{align}

\begin{itemize} 
\item[--] To this end, we consider the auxiliary problem: given 
\begin{align}\label{B-1} 
	g \in  L^q(\B) , \ \  \int_{\B} g \dx = 0 ,
\end{align}
find a vector field $\mathbf v= \LL_b [g]$ such that 
\begin{align}\label{B-2}
	\mathbf v \in W^{1,q}_0(\B; \mathbb R^3) , \ \  \Div_x \mathbf v = g \ \ \text{in } \B.
\end{align}

\item[--] The problem \eqref{B-1}--\eqref{B-2} admits many solutions, but we use the construction due to Bogovskii \cite{Bogovskii}. In particular, we write the following results from \cite[Galdi, Chapter III.3]{Galdi}.

\item[--] The operator $\LL_b$ is bounded, linear and enjoys the bound 
\begin{align}\label{Bogovskii-esti-1}
\|\LL_b [g] \|_{W^{1,q}_0(\B; \R^3)} \leq c(q) \|g\|_{L^q(\B)} \ \text{ for any } 1<q<\infty.
\end{align} 
If, moreover, $g$ can be written in the form $g=\Div_x \mathbf G$ for a certain $\mathbf G\in L^r(\B; \mathbb R^3)$, $\mathbf G \cdot \mathbf n |_{\partial \B}= 0$, then 
\begin{align}\label{Bogovskii-esti-2}
\|\LL_b [g] \|_{L^{r}(\B; \mathbb R^3)} \leq c(r) \|\mathbf G\|_{L^{r}(\B;\mathbb R^3)} \ \text{ for any } 1<r<\infty   .
\end{align}
\end{itemize}

Now, let us consider the following function 
\begin{align}\label{tst-func-Bog}
\boldvphi(t,x) = \psi_K(t,x)  \LL_b [\rho-c(\rho_0)]  \ \ \ \forall (t,x) \in (0,T) \times \B
\end{align} 
with    $c(\rho_0)= \frac{1}{|\B|} \int_{\B} \rho \dx$,  and
\begin{align*}
 \psi_K \in \C_c^\infty((0,T)\times  \B)) \text{ such that } 0\leq \psi_K \leq  1 \text{ in } K \text{ and } \psi_K \equiv 0 \text{ in } ((0,T)\times \B) \setminus K ,  
\end{align*} 
where $K\subset (0,T)\times \B$ is any compact set fulfilling the property \eqref{Set-K}. 

We use this $\boldvphi$ as a test function in the momentum equation \eqref{weak-momen}. In what follows, we get (after some  steps of computations)   
In what follows, we have 
\begin{align}\label{modi-pres-frm-momen}
&\iint_K   p_{\xi, \delta}(\rho, \theta)\rho \, \psi_K \notag  \\
&= c(\rho_0) \iint_K  p_{\xi, \delta}(\rho, \theta)  \psi_K  - \iint_K p_{\xi, \delta}(\rho, \theta) \nabla_x \psi_K \cdot \LL_b [\rho-c(\rho_0)]
\notag \\
& \ \  - \iint_K \partial_t \psi_K  \rho \uvect    \LL_b [ \rho - c(\rho_0)  ]  
 +    
\iint_K \rho \uvect    \LL_b [ \Div_x (\rho \uvect) ] \psi_K \notag \\ 
 & \ \ - \iint_K \rho [\uvect \otimes \uvect] : \nabla_x \left(\psi_K \LL_b [ \rho - c(\rho_0)]\right)   + 
\iint_K
\Ss_\omega : \nabla_x\left(\psi_K \LL_b [ \rho - c(\rho_0) ]\right) \notag \\ 
& =  \sum_{j=1}^6 I_j .
\end{align}

\begin{itemize} 
\item[--] 
Recall that $p_{\xi, \delta}(\rho, \theta)= p_{M}(\rho, \theta) + \frac{a_\xi}{3}\theta^4 + \delta \rho^\beta$,  where $p_M$ satisfies \eqref{mole_bound_p_M}. 

By means of the estimates  \eqref{uniform-bound-2}, \eqref{bound-theta-L4} and \eqref{rho_5/3}, we get 
\begin{align}\label{esti-I-1}
|I_1| &= \bigg| c(\rho_0) \iint_K  p_{\xi, \delta}(\rho, \theta) \psi_K    \bigg| \notag \\
&\leq c(\rho_0) \iint_K |\psi_K|   \left( \rho^{\frac{5}{3}} + \theta^{\frac{5}{2}} + a_\xi\theta^4 + \delta \rho^\beta \right)    \bigg| \leq C(\lambda, \xi)   .
\end{align} 

\vspace*{.1cm}
\item[--] Next, we find that
\begin{align}
|I_2| &= \bigg|   \iint_K p_{\xi, \delta} (\rho, \theta) \nabla_x \psi_K \cdot \LL_b [\rho - c(\rho_0)] \bigg|   \notag  \\
& \leq \|\nabla_x\psi_K\|_{L^\infty(K)} \int_0^T 
\int_\B \left( \rho^{\frac{5}{3}} + \theta^{\frac{5}{2}} + a_\xi\theta^4 + \delta \rho^\beta \right) \left|\LL_b[\rho-c(\rho_0)]\right| \notag \\
& \leq C(\xi) \int_0^T 
 \left( 1 + \|\rho^{\frac{5}{3}}\|_{L^1(\B)} + \|a_\xi\theta^4\|_{L^1(\B)} + \|\delta \rho^\beta\|_{L^1(\B)}  \right) \|\LL_b [\rho-c(\rho_0)]\|_{L^\infty(\B; \R^3)} \notag \\
&\leq C(\lambda, \xi) \int_0^T  \|\LL_b [\rho-c(\rho_0)]\|_{W^{1,4}_0(\B; \R^3)}  \leq C(\lambda, \xi) \|\rho \|_{L^\infty(0,T; L^4(\B))} \notag \\
&\leq C(\lambda, \xi, \delta),
\end{align} 
thanks to the embedding $W^{1,4}_0(\B)\hookrightarrow L^\infty(\B)$ (since we are in dimension $3$), the relation \eqref{Bogovskii-esti-1}, and the bounds \eqref{uniform-bound-2}, \eqref{bound-theta-L4} and \eqref{rho_5/3}. Here we also used the fact that $\beta \geq 4$ to ensure a proper bound of $\|\rho\|_{L^\infty(0,T; L^4(\B))}$. 

\vspace*{.1cm}

\item[--] By H\"older inequality, we get 
\begin{align}
|I_3|  & = \bigg|\iint_K \partial_t \psi_K  \rho \uvect  \LL_b \left[ \rho - c(\rho_0)  \right] \bigg| \notag \\
&\leq  \|\partial_t\psi_K\|_{L^\infty(K)} \int_0^T  \|\sqrt{\rho} \|_{L^2(\B)} \| \sqrt{\rho} \uvect \|_{L^2(\B; \R^3)} \| \LL_b\left[ \rho - c(\rho_0) \right] \|_{L^\infty(\B; \R^3)}  \notag \\
& \leq C(\lambda) \int_0^T \| \LL_b\left[ \rho - c(\rho_0) \right] \|_{W^{1,4}_0(\B; \mathbb R^3)} \notag \\
& \leq C \int_0^T \| \rho  \|_{L^{4}(\B)} \leq C(\lambda, \delta),
\end{align} 
where we have used \eqref{Bogovskii-esti-1}, the fact that $\beta\geq 4$ and the estimates \eqref{uniform-bound-2}, \eqref{uniform-bound-3} and  \eqref{rho_5/3}.  

\vspace*{.1cm}

\item[--]  By virtue of \eqref{Bogovskii-esti-2}, \eqref{uniform-bound-2} and \eqref{uniform-bound-5}, we get 
\begin{align}
|I_4| &= \bigg|   \iint_{K} \psi_K \rho \uvect \LL_b \left[ \Div_x (\rho \uvect)  \right]     \bigg|
\notag \\
&\leq \|\psi_K\|_{L^\infty(K)} \int_0^T  \| \rho \uvect\|_{L^2(\B; \R^3)} \|\LL_b \left[\Div_x (\rho \uvect)  \right] \|_{L^2(\B; \R^3)}  \notag \\
& \leq C \int_0^T \|\rho \uvect\|^2_{L^2(\B;\R^3)} \leq \int_0^T \|\uvect \|^2_{L^6(\B;\R^3)} \| \rho \|^2_{L^3(\B)}  \notag \\
&\leq 
C \|\rho\|^2_{L^\infty(0,T; L^3(\B) )} \|\uvect \|^2_{L^2(0,T; L^6(\B; \R^3) )} 
\leq C(\lambda, \omega, \delta) .
 \end{align}

\vspace*{.1cm}

\item[--] The fifth term in \eqref{modi-pres-frm-momen} can be estimated as
\begin{align}
|I_5| &= \bigg|   \iint_{K} \rho [\uvect \otimes \uvect] : \nabla_x \left(\psi_K \LL_b \left[ \rho - c(\rho_0) \right] \right)  \bigg| 
\notag \\
& \leq \int_0^T \|\rho[\uvect \otimes \uvect]\|_{L^{\frac{3}{2}}(\B; \R^{3\times 3}) } \| \nabla_x \left(\psi_K \LL_b [\rho - c(\rho_0)] \right) \|_{L^{3}(\B ; \R^3)}   \notag \\
&  \leq C \int_0^T  \|\rho |\uvect|^2 \|_{L^{\frac{3}{2}}(\B; \R^3)}  \|\LL_b [\rho-c(\rho_0)]\|_{W^{1,3}_0(\B;\mathbb R^3)} \notag \\
& \leq C \int_0^T  \|\rho\|_{L^3(\B)} \|\uvect\|^2_{L^6(\B; \R^3)} \|\rho\|_{L^3(\B)}  \leq C(\lambda, \omega, \delta) ,
 \end{align} 
by using H\"older inequality, and the relations \eqref{Bogovskii-esti-1}, \eqref{uniform-bound-2} and \eqref{uniform-bound-5}.

\vspace*{.1cm}

\item[--] Finally,  we compute 
\begin{align}\label{esti-I-6}
|I_6| &= \bigg|   \iint_{K} \Ss_\omega : \nabla_x\left(\psi_K \LL_b [\rho - c(\rho_0)] \right) \bigg| \notag \\
&\leq \int_0^T \|\Ss_\omega\|_{L^{\frac{5}{4}}(\B; \R^3)} \|\nabla_x \left(\psi_K\LL_b [\rho-c(\rho_0)]\right) \|_{L^4(\B; \R^3)} \notag \\
& \leq C\int_0^T \|\Ss_\omega\|_{L^{\frac{5}{4}}(\B; \R^3)} \|\LL_b[\rho-c(\rho_0)] \|_{W^{1,4}_0(\B;\R^3)} \notag \\
&\leq C\|\Ss_\omega\|_{L^{\frac{5}{4}}((0,T)\times \B )} \| \rho \|_{L^4((0,T)\times \B)} \leq C(\lambda, \nu, \xi, \delta) .
\end{align} 
In above, we have used the bound \eqref{uniform-bound-2} and the following:  observe that 
\begin{align*}
\int_0^T \int_{\B} |\Ss_\omega |^{\frac{5}{4}} &\leq C \int_0^T \int_\B \bigg|\frac{1}{\sqrt{\theta}} \sqrt{\Ss_\omega : \nabla_x \uvect} \  \theta   \bigg|^{\frac{5}{4}} \\
& \leq C\left(\int_0^T \int_\B 
\bigg|\frac{1}{\sqrt{\theta}} \sqrt{\Ss_\omega : \nabla_x \uvect}    \bigg|^{2} \right)^{\frac{5}{8}} \left(\int_0^T \int_\B \theta^{\frac{10}{3}}\right)^{\frac{3}{8}},
\end{align*}
and thus, 
\begin{align}\label{bound-Ss-omega}
\|\Ss_\omega \|_{L^{\frac{5}{4}}((0,T)\times \B) } \leq \left( \int_0^T \int_\B \frac{1}{\theta} \Ss_\omega : \nabla \uvect \right)^{\frac{1}{2}} \|\theta \|_{L^{\frac{10}{3}}((0,T)\times \B)} \leq C(\lambda, \nu, \xi) ,
\end{align}
thanks to the bounds \eqref{uniformbound-Stress}, \eqref{bound-theta-L4} and  \eqref{bound-theta-gamma}.

\end{itemize}

\vspace*{.1cm}

Summing up the estimates \eqref{esti-I-1}--\eqref{esti-I-6} in \eqref{modi-pres-frm-momen}, we obtain the required bound \eqref{Bogovskii}.


\begin{remark} 
Here, we must mention that,   due to the fact that the boundaries $\Gamma_\tau$ change with respect to time, the uniform estimate of  like \eqref{Bogovskii} on
the whole space-time cylinder $(0,T)\times \B$ seems to be a delicate matter. On the other hand, the mere
equi-integrability of the pressure could be shown by the method based on special test functions used
in \cite{FEIREISL-JIRI}. {\color{red} More details can be found in the monograph \cite{KMNPW-L}.}
\end{remark}

\vspace*{.1cm}
\noindent
(ix)     Using the classical maximum principle to the equations of species mass fractions
\begin{align}
\rho \left(\partial_t Y_k + \uvect \cdot \nabla_x Y_k      \right) = \Div_x(\zeta_\omega(\theta) \nabla_x Y_k) + \rho \sigma_k \quad \text{in } (0,T) \times \B
\end{align}
with the initial data \eqref{modified-mass-initial}, one has 
\begin{align}\label{L-infinity-bound-Y}
\| Y_k \|_{L^\infty((0,T)\times \B)} \leq C 
\end{align}
(see, for instance, \cite[Section 5.2]{Eduard-CPAA-2008}), where $C>0$ does not depend on any parameters. 

Moreover, by means of \eqref{eq-mass-specs-2-pena}, one has
\begin{align}
\label{regular-bound-Y}
\int_0^T \int_{\B} \zeta_\omega(\theta) |\nabla_x Y_k|^2 \leq C,
\end{align} 
where the constant $C>0$ is independent of $\veps$ and other parameters.  

Using the expression of $\zeta_\omega(\theta)$ in \eqref{def-sigma-omega} and the fact $\zeta(\theta)> \underline{\zeta}>0$  from \eqref{hypo-zeta}, we further have  
\begin{align}\label{bound-nabla-Y-k} 
c_2(\omega)\|\nabla_x Y_k \|^2_{L^2( 0,T; L^2(\B) )} \leq C ,  
\end{align} 
where the constant $c_2(\omega)>0$
	behaves like $c \, \omega$ in $((0,T)\times\B)\setminus Q_T$  ($c>0$ is independent in $\omega$), thanks to the construction of $\zeta_\omega(\theta)$ in \eqref{def-sigma-omega}. 

Furthermore, by using the bounds \eqref{uniform-bound-4} and \eqref{regular-bound-Y}, we get
\begin{align}\label{bound-L-1-zeta-Y}
  \int_0^T \int_{\B} \zeta_\omega(\theta) |\nabla_x Y_k| \leq \frac{1}{2} \int_0^T \int_{\B} \zeta_\omega(\theta) |\nabla_x Y_k|^2 + \frac{1}{2} \int_0^T \int_\B \zeta_\omega(\theta) \leq C (\lambda).
\end{align} 

\vspace*{.2cm}
\noindent 
(x)  We now need to find  proper bounds for  $\disp \rho s_{\xi}$ and $\disp \rho s_\xi \uvect$. 
Recall \eqref{third-law}, and thus there exists some $c>0$ such that 
\begin{align}\label{Fact-1}  
s_M(\rho,\theta)	= S\left(\frac{\rho}{\theta^{\frac{3}{2}}}\right)  \leq c  \quad \text{when } \frac{\rho}{\theta^{3/2}} > 1 ,
\end{align}
and therefore, 
\begin{align}\label{Fact-2} 
\rho \left(s_M (\rho, \theta) + s_{R,\xi}(\rho, \theta)\right) = \rho s_M(\rho, \theta) + \frac{4a_\xi}{3}\theta^3 \leq c\rho + \frac{4a_\xi}{3}\theta^3 \quad \text{for } \frac{\rho}{\theta^{3/2}} >1 .
\end{align}

On the other hand, when $\disp \frac{\rho}{\theta^{3/2}} \leq 1$, we use the strategy developed in \cite[Section 4, formula (4.6)]{Feireisl2012weak} and  according to that, one has  (using  Gibb's relation \eqref{equ-gibbs}  and \eqref{Gibbs-1})
\begin{align*}
s_M(\rho,\theta) \leq 	C (1+ |\log \rho| + [\log \theta]^+) .
\end{align*}
This yields
\begin{equation} 
\begin{aligned}\label{Fact-3-1}
	\rho \left(s_M (\rho, \theta) + s_{R,\xi}(\rho, \theta)\right) = \rho s_M(\rho, \theta) + \frac{4a_\xi}{3}\theta^3
	\leq 	C	\left(\rho + |\rho \log \rho|  + |\rho| [\log \theta]^+ \right) + \frac{4a_\xi}{3}\theta^3  .  
\end{aligned}
\end{equation}
Now, observe that 
\begin{align}\label{bound-rho-log-rho}
|\rho \log \rho| \leq 
\begin{dcases}
	C \rho^{\frac{1}{2}} \ \ \ \text{when } 0< \rho \leq 1, \\
	\frac{3}{2} \rho [\log \theta]^+ \ \ \ \text{when } \rho >1 \ \ (\text{consequently $\theta>1$ since $\frac{\rho}{\theta^{3/2}} \leq 1$}),
\end{dcases}
\end{align} 
where we have used the fact that 
$|\rho^{\frac{1}{2}} \log \rho|$ is bounded for $0<\rho \leq 1$.

Using \eqref{bound-rho-log-rho} in \eqref{Fact-3-1}, we get
\begin{align}\label{bound-rho_s}
\rho \left(s_M (\rho, \theta) + s_{R,\xi}(\rho, \theta)\right)  \leq C \left(  \rho + \rho^{\frac{1}{2}} + \theta^{\frac{3}{2}}[\log \theta]^+ \right) + \frac{4a_\xi}{3}\theta^3
 \quad  \text{ for } \frac{\rho}{\theta^{3/2}} \leq 1 .
\end{align}

Now, recall the form of $s_\xi(\rho, \theta, \Yvect)$ from \eqref{exp-new-s}. In what follows,  the relation \eqref{bound-rho_s} together with  \eqref{uniform-bound-4},  \eqref{rho_5/3} and  \eqref{L-infinity-bound-Y}, we can ensure that  
\begin{align}\label{rho-s-xi}    
\|\rho s_{\xi} \|_{L^q((0,T)\times \B)} \leq C(\lambda)  \quad \text{for certain } q>1,
 \end{align}
where the constant $C(\lambda)>0$ does not depend on  
$\veps$. 

One can  further deduce that 
\begin{align}\label{rho-s-xi-u}
\|\rho s_{\xi} \uvect\|_{L^1((0,T)\times \B)} \leq C(\lambda)  ,
 \end{align}
and again the constant $C(\lambda)>0$ is independent   in  
$\veps$. This can be determined  as follows: the crucial terms to estimate are $a_\xi\theta^3\uvect$ and $\rho Y_k \uvect$ (due  to the formulation of $s_\xi(\rho, \theta, \Yvect)$ and \eqref{bound-rho_s}).   Indeed, we find 
\begin{align*}
\left|\int_0^T \int_\B a_\xi  \theta^3 \uvect \right| 
&\leq C \|\theta^3\|_{L^\infty(0,T; L^{\frac{4}{3}}(\B))} \|\uvect\|_{L^1(0,T; L^4(\B; \mathbb R^3))}\\
& \leq C \left\| \theta^4 \right\|^{\frac{3}{4}}_{L^\infty(0,T; L^1(\B) )} \left\|\uvect \right\|_{L^2(0,T; W^{1,2}_0(\B; \mathbb R^3))},
\end{align*} 
which is uniformly  bounded in $\veps>0$  by means of \eqref{uniform-bound-5} and \eqref{bound-theta-L4}. 

On the other hand,  we have 
\begin{align*}
\left|\int_0^T \int_\B  \rho Y_k \uvect \right| \leq \|Y_k \|_{L^\infty((0,T)\times \B)} \left(  C(\rho_0) + \int_0^T \int_\B \rho |\uvect|^2 \right) , 
\end{align*}
which is again bounded uniformly w.r.t. $\veps>0$ by using \eqref{L-infinity-bound-Y} and \eqref{uniform-bound-3}. 

As a consequence,  the estimate \eqref{rho-s-xi-u} holds true.

\section{Penalization limit: passing with $\veps \to 0$} 


\subsection{A first set of limits}  

Let us  fix all the parameters  $\lambda$, $\nu$, $\xi$,   $\omega$ and $\delta$.   

\vspace*{.2cm}
\noindent 
(i) Passing to the limit $\veps\to 0$, we directly obtain
\begin{align}\label{retriving boundary}
	(\uvect - \V )\cdot \mathbf n \big|_{\Gamma_\tau} = 0 \quad \text{for a.a. } \ \tau \in [0,T],  
\end{align}
so that we can retrieve the impermeability boundary condition \eqref{imperm}.

\vspace*{.2cm} 
\noindent 
(ii)
By \eqref{bound-theta-L4}, \eqref{rho_5/3} and \eqref{bound-theta-gamma}, we respectively have (up to a suitable subsequence)
\begin{align}
&	\theta_\veps \to \theta \quad \text{weakly$^*$ in }   \ L^\infty(0,T; L^4(\B))  \ \text{ as $\veps \to 0$}   , \label{conv-theta-1}   \\
&   \rho_\veps \to \rho \quad \text{weakly$^*$ in }   \ L^\infty(0,T; L^{\frac{5}{3}}(\B)) \ \text{ as $\veps \to 0$}  \, \text{ and} \label{conv-rho-1}   \\   
& \theta_\veps \to \theta \quad \text{weakly in } \  L^2(0,T; W^{1,2}(\B)) \ \text{ as $\veps \to 0$}   .  \label{conv-theta-2} 
\end{align}

\vspace*{.2cm}
\noindent 
(iii)
Due to  \eqref{uniform-bound-4} and \eqref{bound-theta-L4}, we also have 
\begin{align}
\label{conv-theta-8}
	&\theta^5_\veps  \to \overline{\theta^5} \quad \text{weakly in } \ L^1((0,T)\times \B) \ \text{ as $\veps \to 0$} \, \text{ and} \\
	\label{conv-theta-4}
	&	\theta^4_\veps  \to \overline{\theta^4} \quad \text{weakly in } \ L^1((0,T)\times \B) \ \text{ as } \veps \to 0 .
\end{align}
Here and in the sequel, the ``bar" denotes a weak limit of a composed or nonlinear function. 

\vspace*{.2cm} 
\noindent 
(iv)
Thanks to  \eqref{uniform-bound-5},
we have 
\begin{align}
	\label{conv-u}
	\uvect_\veps  \to \uvect \quad \text{weakly in } \ L^2(0,T; W^{1,2}_0(\B; \mathbb R^3)) \ \ \text{as }\veps \to 0 .
\end{align}

\vspace*{.2cm}
\noindent 
(v) 
Moreover, we have better convergence result of  $\{\rho_\veps\}_{\veps>0}$ than \eqref{conv-rho-1}.  One can get 
\begin{align}\label{conv-rho-2} 
 \rho_\veps \to \rho \quad \text{in }  \ \C_{\text{weak}}([0,T]; L^{\frac{5}{3}}(\B)) \ \text{ as }\veps \to 0 .  
\end{align}   
Indeed, from the continuity equation \eqref{continuity-eq} and the bound \eqref{bound-rho-u-5/4}, we can check that the functions
\begin{align*}
\left\{t \mapsto \bigg(\int_{\B} \rho_\veps \phi\bigg)(t) \right\} \, \text{ are equi-continuous and bounded in } \C^0([0,T]) \ \text{for any } \phi \in \C^\infty_c(\B).  
\end{align*} 
  Consequently, by standard Arzel\`a-Ascoli theorem, we have 
\begin{align*}
\int_\B \rho_\veps \phi 
\to \int_\B  \rho \phi \ \  \text{in } \C^0([0,T]) \ \text{for any } \phi \in \C^\infty_c(\B)   .
\end{align*} 
Since $\rho_\veps$ satisfies 
the bound \eqref{rho_5/3}, the above convergence can be extended  for each $\phi \in L^{\frac{5}{2}}(\B)$ via density argument and thus, the limit \eqref{conv-rho-2} holds true.

\vspace*{.2cm}
\noindent 
(vi)
Using the limits \eqref{conv-u}, \eqref{conv-rho-2}, the bound \eqref{bound-rho-u-5/4} and the fact that  $L^{\frac{5}{3}}(\B) \hookrightarrow W^{-1,2}(\B)$ is compact, one has 
\begin{align}\label{conv-rho-u}
	\rho_\veps \uvect_\veps \to \rho \uvect \quad \text{weakly$^*$ in } \ L^{\infty}(0,T; L^{\frac{5}{4}}(\B; \mathbb R^3))  \ \text{ as }\veps \to 0.
\end{align}

\vspace*{.2cm} 
\noindent 
(vii)  On the other hand, thanks to the bound \eqref{bound-convective-term}, we have 
%
\begin{align}\label{weak-limit-rho_u-cross_u}
\rho_\veps \uvect_\veps \otimes \uvect_\veps \to  \overline{\rho \uvect \otimes \uvect} \quad \text{weakly in } L^2(0,T; L^{\frac{30}{29}}(\B; \mathbb R^3)) .
\end{align}
Now, our aim is to show that 
\begin{align}\label{pointwise-convective}
\overline{\rho \uvect \otimes \uvect} = \rho \uvect \otimes \uvect \quad \text{a.a. in } (0,T)\times \B . 
\end{align}
To show this, let us first consider any space-time cylinder $(T_1, T_2)\times \mathcal O \subset [0,T]\times \B$ 
 such that 
 \begin{align}\label{condi-space-time}
[T_1, T_2] \times \overline{\mathcal O} \cap \left( \cup_{\tau \in [0,T]} ( \{\tau\} \times \Gamma_\tau ) \right) = \emptyset.  
 \end{align} 
Then, from the momentum equation \eqref{momentum-eq} and the uniform   bounds (w.r.t. $\veps>0$) of the convective term $\rho_\veps [\uvect_\veps\otimes \uvect_\veps]$ from \eqref{bound-convective-term},   pressure term $p_{\xi, \delta}(\rho_\veps, \theta_\veps)$ given by \eqref{Bogovskii} and the bound of $\Ss_\omega(\nabla_x \uvect_\veps)$ from \eqref{bound-Ss-omega},  one can find  that 
the functions 
\begin{align*}
\left\{t \mapsto \int_{\mathcal O} \rho_\veps \uvect_\veps \cdot \boldphi    \right\} \text{ are equi-continuous and bounded in } \C^0([T_1,T_2]) \text{ for any } \boldphi \in \C_c^\infty(\mathcal O; \R^3).
\end{align*} 
Consequently, by Arzel\`{a}-Ascoli theorem, we deduce  that 
\begin{align*}
\rho_\veps \uvect_\veps \to \rho \uvect \ \text{ in } \, \C_{\textnormal{weak}}([T_1,T_2]; L^{\frac{5}{4}}(\mathcal O; \mathbb R^3)) , 
\end{align*} 
for any space-time cylinder $(T_1,T_2)\times \mathcal O\subset [0,T]\times \B$ satisfying \eqref{condi-space-time}. 

On the other hand, since $L^{\frac{5}{4}}(\mathcal O)$ is compactly embedded into $W^{-1,2}(\mathcal O)$, we infer that 
\begin{align*}  
\rho_\veps \uvect_{\veps} \to \rho \uvect \ \text{ strongly in } \, \C_{\textnormal{weak}}([T_1,T_2]; W^{-1,2}(\mathcal O; \R^3) ) .
\end{align*}
%

The above limit, together with the weak convergence of the velocities $\{\uvect_\veps\}_{\veps>0}$ in \eqref{conv-u} give rise to \eqref{pointwise-convective}.

\vspace*{.2cm} 
\noindent 
(viii) Finally, by virtue of \eqref{L-infinity-bound-Y} and \eqref{regular-bound-Y}, we have 
\begin{align}\label{limits-Y}
\begin{dcases} 
Y_{k,\veps} \to Y_k \ \text{ weakly$^*$ in } L^\infty((0,T) \times \B), \text{ and}\\
Y_{k, \veps} \to Y_k  \ \text{ weakly in } L^2(0,T; W^{1,2}(\B)) 
\end{dcases}
\end{align} 
as $\veps\to 0$ (up to a suitable subsequence, if necessary) for $k=1,...,n$. 

\vspace*{.1cm}

We can even achieve strong convergence of the sequence $\{Y_{k,\veps}\}_{\veps>0}$ on the set $\{\rho >0\}$, more precisely
\begin{align}\label{limit-rho-Y-k} 
\int_0^T \int_\B \rho |Y_{k,\veps}|^2 \to \int_0^T \int_\B \rho |Y_k|^2   \ \ \ \forall k=1,\cdots, n.
\end{align} 
Indeed, by  using \eqref{limits-Y}, \eqref{conv-rho-2} and \eqref{eq-mass-species-1-pena}, one has
\begin{align}\label{limits-rho-Y}
\begin{dcases} 
 \rho_\veps Y_{k, \veps} \to \rho Y_k \ \text{ in } \,\C_{\text{weak}}([0,T]; L^{\frac{5}{3}}(\B)) , \\
\rho_\veps|Y_{k,\veps}|^2 \to \rho |Y_k|^2  \ \text{ weakly$^*$ in } \,  L^\infty(0,T; L^{\frac{5}{3}}(\B)),   \text{ and}\\
(\rho_\veps - \rho)|Y_{k,\veps}|^2 \to  0 \ \text{ weakly$^*$ in } \, L^\infty(0,T; L^{\frac{5}{3}}(\B) ) 
\end{dcases}
\end{align} 
for each $k=1, ... , n$. 

In fact, the limit \eqref{limit-rho-Y-k} yields 
\begin{align}\label{pointwise-conv-Y} 
Y_{k,\veps} \to Y_k \ \text{ a.e. in the set } \, \{\rho >0 \} \subset (0,T) \times \B  \ \text{ for } k=1, ... , n .
\end{align}

\subsection{Pointwise convergence of the temperature and the density} 

(i) In order to show a.e. convergence of
the temperature, we follow the technique  based on the Div-Curl Lemma  (see Tartar \cite{Tartar1979compensated}) and
Young measures methods (see Pedregal \cite{Pedregal1997parametrized}), which has been discussed for instance in  \cite[Section 3.6.2]{Feireisl-Novotny-book} at length. To this end,       
we  set 
\begin{align}
&	\mathbf U_\veps = \left[\rho_\veps (s_M+s_{R,\xi}) (\rho_\veps , \theta_\veps ), \  \rho_\veps (s_M+s_{R,\xi})(\rho_\veps, \theta_\veps) \uvect_\veps + \frac{\kappa_\nu(\theta_\veps) \nabla_x \theta_\veps}{\theta_\veps}   \right], \\
& \mathbf W_\veps = \left[G(\theta_\veps), 0, 0, 0  \right] , 
\end{align}
where $G$ is a bounded and globally Lipschitz function in 
$[0, \infty)$.  Then due to the estimates obtained in previous section, $\Div_{t,x} \mathbf U_\veps$ is precompact in $W^{-1,s}((0,T)\times \B)$ and $\text{Curl}_{t,x} \mathbf W_\veps$ is precompact in $W^{-1,s}((0,T)\times \B)^{4\times 4}$ with certain $s>1$. Therefore, using the Div-Curl lemma 
for $\mathbf U_\veps$ and $\mathbf W_\veps$, we may derive that 
\begin{align}\label{ineq-1}
\overline{\rho (s_M+ s_{R,\xi})(\rho, \theta) G(\theta)} = \overline{\rho (s_M+ s_{R,\xi})(\rho, \theta) }\ \overline{G(\theta)} .
\end{align}
In fact, by applying the theory of parameterized (Young) measures (see \cite[Section 3.6.2]{Feireisl-Novotny-book}), one can show that
\begin{align}\label{ineq-2} 
\overline{\rho s_M(\rho, \theta) G(\theta)} \geq 	\overline{\rho s_M(\rho, \theta)} \ \overline{ G(\theta)} \ \text{ and } \ \overline{\theta^3 G(\theta)} \geq \overline{\theta^3} \ \overline{G(\theta)} .
\end{align}
Combining \eqref{ineq-1}--\eqref{ineq-2} and taking $G(\theta)= \theta$, we deduce 
\begin{align*}
\overline{\theta^4} = \overline{\theta^3} \, \theta  ,
\end{align*}
which  yields 
\begin{align}\label{limit-theta}
\theta_\veps \to \theta  \quad \text{a.a. in } \ (0,T)\times \B. 
\end{align}
Moreover,  thanks to \eqref{bound-log-theta} and  using the generalized Poincar\'{e} inequality in \Cref{Poincare}, one can prove that  $\log \theta \in L^2((0,T) \times \B)$ which ensures that the  limit temperature is positive a.e. on the set $(0,T)\times \B$.

\vspace*{.2cm}
\noindent 
(ii)  Next, we show pointwise convergence of the sequence $\{\rho_\veps\}_{\veps>0}$ in $(0,T)\times \B$. 

We define   $T_k(\rho) = \min\{\rho, k\}$. 
Similar to the analysis in \cite{Feireisl-NSF-1}, one can find the {\em effective viscous pressure identity}:
\begin{align}\label{effective-idendity}
\overline{p_{\xi,\delta}(\rho) T_k(\rho)} - \overline{p_{\xi,\delta}(\rho)} \ \overline{T_k(\rho)}  = \left(\frac{4}{3} \mu_\omega + \eta_\omega\right) \left( \overline{T_k(\rho) \Div_x \uvect} - \overline{T_{k}(\rho)} \, \Div_x \uvect \right) ,
\end{align} 
which holds on each compact set $K \subset (0,T)\times \B$ satisfying 
\begin{align*}
K \cap \left( \cup_{\tau \in [0,T]} (\{\tau\} \times \Gamma_\tau) \right). 
\end{align*}

Now, following \cite{Eduard-CPAA-2008} (see also  \cite[Chapter 6]{Feireisl-NSF-2} and \cite{Feireisl-NSF-1}), we introduce the
 {\em oscillations defect measure} 
 \begin{align*}
\textbf{osc}_q [\rho_\veps \to \rho](K) = \sup_{k\geq 0} \left( \limsup_{\veps \to 0} \iint_K \left|T_k(\rho_\veps) - T_k(\rho) \right|^q \right) ,
 \end{align*} 
 and then use \eqref{effective-idendity} (as well as \eqref{Bogovskii}) to conclude that 
 \begin{align*}
\textbf{osc}_{\frac{8}{3}}[\rho_\veps \to \rho](K)  \leq C(\omega), 
 \end{align*} 
where the constant $C(\omega)$ is independent in $\veps$ and the sets $K$. Thus, 
\begin{align*}
\textbf{osc}_{\frac{8}{3}}[\rho_\veps \to \rho]( (0,T)\times \B )  \leq C(\omega) . 
 \end{align*} 
This implies the desired conclusion 
\begin{align}\label{limit-rho}
\rho_\veps \to \rho \, \text{ a.e. in } (0,T) \times \B  .  
\end{align} 
by virtue of the procedure developed in \cite[Section 7]{Eduard-CPAA-2008} (we also refer \cite[Chapter 6]{Feireisl-NSF-2}).

\vspace*{.2cm}

\begin{itemize} 
\item Then, by using  \eqref{conv-u} and \eqref{limit-theta}, one can identify 
\begin{align}\label{limit-stress}
\Ss_\omega(\theta_\veps, \nabla_x \uvect_\veps) \to \Ss_\omega(\theta, \nabla_x \uvect) \  \  &\text{ weakly in } \ L^1((0,T)\times \B) .
\end{align} 
Moreover, thanks to the bounds \eqref{rho-s-xi}, \eqref{rho-s-xi-u} and the limits \eqref{conv-u}, \eqref{pointwise-conv-Y}, \eqref{limit-theta} and \eqref{limit-rho}, we achieve that  
\begin{align}\label{limits-stress-entropy}
\begin{dcases}
\rho_\veps  s_\xi
(\rho_\veps, \theta_\veps) \to \rho s_{\xi}(\rho, \theta) \ \ &\text{ weakly in } \ L^1((0,T)\times \B) \, \text{ and} \\
    \rho_\veps s_\xi(\rho_\veps, \theta_\veps) \uvect_\veps \to \rho s_\xi (\rho, \theta) \uvect \ \ &\text{ weakly in } \ L^1((0,T)\times \B) 
\end{dcases}
\end{align}
up to a suitable subsequence. 
In particular, one can show that 
\begin{align}\label{limit-rho-s-M-s-R}
\begin{dcases}
\rho_\veps  (s_M + s_{R,\xi})
(\rho_\veps, \theta_\veps) \to \rho (s_M + s_{R,\xi})(\rho, \theta) \ \ &\text{ weakly in } \ L^1((0,T)\times \B),  \\
    \rho_\veps (s_M + s_{R,\xi})(\rho_\veps, \theta_\veps) \uvect_\veps \to \rho (s_M + s_{R,\xi})(\rho, \theta) \uvect \ \ &\text{ weakly in } \ L^1((0,T)\times \B) ,
\end{dcases}
\end{align} 
and 
\begin{align}\label{limit-rho-Y-k-u}
\begin{dcases}
\rho_\veps  Y_{k,\veps}
\to \rho Y_k \ \ &\text{ weakly in } \ L^1((0,T)\times \B),  \\
    \rho_\veps Y_{k,\veps} \uvect_\veps \to \rho Y_k \uvect \ \ &\text{ weakly in } \ L^1((0,T)\times \B) .
\end{dcases}
\end{align} 

\vspace*{.2cm}

\item Now, by using the pressure estimate \eqref{Bogovskii}, and the strong convergence results \eqref{limit-theta} and \eqref{limit-rho}, we deduce that
%
%
\begin{align}\label{limit-pressure}
p_{\xi, \delta}(\rho_\veps, \theta_\veps) = p_M(\rho_\veps, \theta_\veps) + \frac{a_\xi}{3} \theta^4_\veps + \delta \rho^\beta_\veps \to    p_{M}(\rho, \theta) + \frac{a_\xi}{3} {\theta^4} + \delta {\rho^\beta}  \text{ weakly in } L^1(K) 
\end{align}
as $\veps \to 0$, where   $K$ is any compact set of $(0,T)\times \B$ satisfying \eqref{Set-K}. 
\end{itemize}

\subsection{The limit system as $\veps\to 0$} In this subsection, we summarize  the limiting  behaviors of all the quantities from the previous two subsections.

\vspace*{.2cm}
\noindent 
(i) Passing to the limit as $\veps\to 0$, one can observe that  the continuity equation satisfies the same integral identity as  \eqref{weak-conti}.

\vspace*{.2cm} 
\noindent 
(ii) Next, we proceed to pass to the limit in the momentum equation \eqref{weak-momen}. Keeping in mind the local estimates (and limits) of the pressure term (see \eqref{Bogovskii} and \eqref{limit-pressure}), we consider the test functions 
\begin{align}\label{test-func}
	\boldvphi \in \C_c^1( [0,T) ; W^{1, \infty}(\B, \mathbb R^3) ) , \ \ \ \text{Supp}\, [\Div_x  \boldvphi(\tau, \cdot)] \cap \Gamma_\tau = \emptyset, \quad \boldvphi \cdot \mathbf n |_{\Gamma_\tau} = 0  \ \ \ \forall \tau \in [0,T] .
\end{align}  
Then, using  the limits in the previous two subsections, we have, upon  $\veps \to 0$ in \eqref{weak-momen}, that
	\begin{align}\label{weak-momen-limit}
		-	\int_0^T \int_{\B} \left( \rho \uvect \cdot \partial_t  \boldvphi   +  \rho[\uvect \otimes \uvect] : \nabla_x  \boldvphi + p_{\xi, \delta}(\rho, \theta) \Div_x \boldvphi   \right)   
		+\int_0^T \int_{\B} \Ss_\omega : \nabla_x \boldvphi \notag  \\
		  =
		\int_{\B}   (\rho \uvect)_{0,\delta}\cdot  \boldvphi(0, \cdot) 
	\end{align}
for any test function as in \eqref{test-func}, where we have also used the boundary  condition \eqref{retriving boundary}.

\vspace*{.2cm}
\noindent 
(iii)  Thanks to the bound \eqref{bound-L-1-zeta-Y} and the limit \eqref{limit-theta},  one can get 
\begin{align}\label{limit-zeta-Y-k}
\zeta_{\omega}(\theta_\veps) \nabla_x Y_{k, \veps} \to  \zeta_{\omega}(\theta) \nabla_x Y_{k} \quad \text{weakly in } \ L^1((0,T)\times \B) . 
\end{align} 
Using \eqref{limit-zeta-Y-k} and the limits    \eqref{limits-Y}$_2$ and \eqref{limit-rho-Y-k-u}  in \eqref{eq-mass-species-1-pena}, the weak formulation for species mass fraction reads as (upon $\veps\to 0$) 
\begin{align}\label{eq-mass-species-limit-eps}
-\int_0^T \int_{\B} \big[\rho Y_k \partial_t \vphi + \rho Y_k \uvect \cdot \nabla_x \vphi  - \zeta_\omega(\theta) \nabla_x Y_k \cdot \nabla_x \vphi   \big] = & \notag \\
		\int_0^T \int_{\B}  \rho \sigma_k \vphi + \int_{\B} \rho_{0,\delta} Y_{k,0, \delta} \vphi(0,\cdot)  \quad &\text{ for  } k=1, ..., n  ,
	\end{align}
which is to be satisfied 	for any test function $\vphi \in \C^1([0,T)\times \B   ;\mathbb R)$ with  $\vphi\geq 0$.  
Moreover, the relation \eqref{eq-mass-specs-2-pena} holds true in the limiting case.

\vspace*{.2cm} 
\noindent 
(iv) Further, by using \eqref{bound-kappa-theta-2} and the limit \eqref{limit-theta}, we have
\begin{align}\label{weak-limit-kappa-theta}
	\frac{\kappa_\nu(\theta_\veps)}{\theta_\veps} \nabla_x \theta_\veps \to  \frac{\kappa_\nu(\theta)}{\theta} \nabla_x \theta \quad \text{weakly in } \ L^1((0,T)\times \B) . 
\end{align}
We also have that the terms $\disp \frac{1}{\theta_\veps} \Ss_\omega(\theta_\veps, \nabla_x \uvect_\veps): \nabla_x \uvect_\veps$ and 
 $\disp \frac{\kappa_\nu(\theta_\veps)|\nabla_x \theta_\veps|^2}{\theta_\veps^2}$ are weakly lower semicontinuous. 


These, together with  the limits 
 \eqref{conv-theta-4},  \eqref{limits-Y}$_2$,   \eqref{limit-theta}, \eqref{limit-rho}, \eqref{limits-stress-entropy} and \eqref{limit-zeta-Y-k}, we obtain from \eqref{penalized-entropy-1} (upon $\veps\to 0$),
\begin{align}\label{limit-enropy-ineq} 
	- & \int_0^T \int_{\B} \bigg(\rho s_\xi \left(\partial_t \vphi +  \uvect \cdot \nabla_x \vphi \right) - \frac{\kappa_\nu(\theta)}{\theta} \nabla_x \theta \cdot \nabla_x \vphi   - \sum_{k=1}^n s_k  \zeta_{\omega}(\theta) \nabla_x Y_k        \cdot \nabla_x \vphi \bigg)  \notag \\
	 & \qquad  - \int_{\B} (\rho s)_{0,\delta} \vphi(0,\cdot)   + \int_0^T \int_\B \lambda \theta^4 \vphi \notag \\
& 	\geq 
	\int_0^T\int_{\B} \frac{\vphi}{\theta}  \left( \Ss_\omega : \nabla_x \uvect  + \frac{\kappa_\nu(\theta)|\nabla_x\theta|^2  }{\theta} - \sum_{k=1}^n \rho (h_k - \theta s_k)\sigma_k \right) 
	\end{align}
for any  test function $\vphi \in \C^1([0,T)\times \B; \mathbb R)$ with $\vphi \geq 0.$

\vspace*{.2cm}
\noindent 
(v)
Let us pass to the limit $\veps\to 0$ in the  energy balance \eqref{energy-pena-1}. Thanks to the almost everywhere convergence results \eqref{limit-theta}, \eqref{limit-rho}, the bound \eqref{bound-rho-e-xi} and the fact that 
$\{\rho_\veps( e_M + e_{R,\xi})(\rho_\veps, \theta_\veps)\}_{\veps}$ is nonnegative, we have, using the Fatou's lemma, 
\begin{align*}
\limsup_{\veps \to 0} \int_0^T \int_{\B} \rho_\veps (e_M+e_{R,\xi})(\rho_\veps, \theta_\veps ) \partial_t \psi \leq \int_0^T \int_{\B} \rho(e_M+ e_{R,\xi})(\rho, \theta) \partial_t \psi
\end{align*} 
as long as  $\psi \in \C^1_c([0,T))$ such that $\psi\geq 0$ and $\partial_t \psi \leq 0.$

Similarly, one has 
\begin{align*}
	\limsup_{\veps \to 0} \int_0^T \int_{\B} \rho_\veps |\uvect_\veps|^2 \partial_t \psi \leq 	 \int_0^T \int_{\B} \rho |\uvect|^2 \partial_t \psi . 
\end{align*}
Using the above information 
and gathering all the limits \eqref{limit-theta}, \eqref{limit-rho},  \eqref{weak-limit-kappa-theta},  \eqref{limit-pressure},  \eqref{weak-limit-rho_u-cross_u}, the limiting  ``energy inequality" reads as (passing to the limit as $\veps\to 0$ in \eqref{energy-pena-1}) 
\begin{align}\label{limit-energy-pena}
&	-	\int_0^T \partial_t \psi \int_{\B}  \left(\frac{1}{2} \rho |\uvect|^2  + \rho e_M(\rho,\theta)  + \rho e_{R,\xi}(\rho, \theta)  + \frac{\delta}{\beta-1} \rho^\beta \right) 
	+ \int_0^T \int_\B \lambda \theta^5 \psi	
		\notag   \\
	 \leq &	-\int_0^T  \psi \int_{\B}  \big(\rho[\uvect \otimes \uvect]    : \nabla_x \V  - \Ss_\omega : \nabla_x \V  + p_{\xi, \delta}(\rho,\theta)\, \Div_x \V  \big) 
		+ \int_0^T \int_{\B} \partial_t(\rho \uvect) \cdot \V  \psi  
		\notag \\ 
 &  -\int_0^T \psi \int_{\B} \sum_{k=1}^n \rho h_k \sigma_k +\psi(0) \int_{\B} 	  \left(\frac{1}{2} \frac{|(\rho \uvect)_{0,\delta}|^2}{\rho_{0,\delta}}  + \rho_{0,\delta} \left(e_M + e_{R,\xi}\right)(\rho_{0,\delta},\theta_{0,\delta})  + \frac{\delta}{\beta-1} \rho_{0,\delta}^\beta \right)
\end{align}
for all $\psi \in \C^1_c([0,T))$ such that $\psi\geq 0$ and  $\partial_t \psi \leq 0$.

\vspace*{.2cm}
\noindent
(vi)  Finally, by using all the above limits and \eqref{limit-rho-s-M-s-R}, the modified energy inequality \eqref{energy-pena-3} becomes 
 \begin{align}\label{energy-pena-3-limit-veps}
 		& \int_{\B}  \left(\frac{1}{2} \rho |\uvect|^2  + \rho( e_M +  e_{R,\xi}) - \rho (s_M +s_{R, \xi})  + \frac{\delta \rho^\beta}{\beta-1}\right)(\tau, \cdot) 
+  \int_0^\tau \int_\B \lambda \theta^5 \notag \\
 	&	  + \int_0^\tau \int_\B \frac{1}{\theta}    \left( \Ss_\omega : \nabla_x \uvect  + \frac{\kappa_\nu(\theta)|\nabla_x\theta|^2  }{\theta} - \sum_{k=1}^n \rho h_k \sigma_k \right) \notag   \\
 	\leq &	-\int_0^\tau  \int_{\B}  \big(\rho[\uvect \otimes \uvect]    : \nabla_x \V  - \Ss_\omega : \nabla_x \V  + p_{\xi, \delta}(\rho,\theta)\, \Div_x \V  + \rho \uvect \cdot \partial_t \V   \big) \notag \\
 		& + \int_\B  (\rho \uvect \cdot \V)(\tau, \cdot) - \int_\B (\rho \uvect)_{0,\delta} \cdot \V(0, \cdot)    - \int_0^\tau \int_{\B} \sum_{k=1}^n \rho h_k \sigma_k  + \int_0^\tau \int_\B\lambda \theta^4     \notag \\ 
 		& + \int_{\B} 	  \left(\frac{1}{2} \frac{|(\rho \uvect)_{0,\delta}|^2}{\rho_{0,\delta}}  + \rho_{0,\delta} (e_M+ e_{R, \xi})(\rho_{0,\delta},\theta_{0,\delta}) - \rho_{0,\delta} (s_M+ s_{R, \xi})(\rho_{0,\delta},\theta_{0,\delta})   + \frac{\delta \rho_{0,\delta}^\beta}{\beta-1}     \right)
 	\end{align}
 for a.a.  $\tau \in [0,T]$.

\section{Get rid of solid part: passing to the limits of other parameters}\label{Section:Solid-part}

Let us take care of the density-dependent terms in the solid part $((0,T)\times \B) \setminus Q_T$. We use \cite[Lemma 4.1]{Sarka-Kreml-Neustupa-Feireisl-Stebel} (see also \cite[Section 4.1.4]{sarka-aneta-al-JMPA}) to conclude that 
the density $\rho$ remains ``zero" on the solid part if it was so initially, thanks to the continuity equation and the fact that density is square-integrable. Here, we must mention that the square-integrability of  density is identified from the estimate of  $\delta \rho^\beta$ given by \eqref{uniform-bound-2}.

We also recall that the parameters $\omega$, $\xi$ and $\nu$ are not  involved in $\Omega_t$, $t\in [0,T]$, by means of our extension strategy given in the beginning of Section \ref{sec-penalized}.

\vspace*{.2cm}
\noindent 
(i) This leads to the following weak formulation for the continuity equation (of course, after passing to the limit as $\veps \to 0$ from previous section), 
\begin{align}\label{weak-conti-recover}
	-\int_0^T \int_{\Omega_t} \rho B(\rho)\left( \partial_t \vphi +    \uvect \cdot \nabla_x \vphi \right) 
	+	\int_0^T \int_{\Omega_t} b(\rho) \Div_x \uvect \vphi 
	=  \int_{\Omega_0} \rho_{0,\delta} B(\rho_{0,\delta})\vphi(0, \cdot)
\end{align}
for any  test function $\vphi \in \C^1_c( [0,T)\times \B; \mathbb R)$  and any $b \in L^\infty \cap \C([0,+\infty))$ such that $b(0)=0$ and $\displaystyle B(\rho) = B(1)+ \int_{1}^\rho \frac{b(z)}{z^2}$. 

\vspace*{.2cm}
\noindent 
(ii) Using the fact $\rho=0$ in $\B\setminus \Omega_t$, the momentum equation reads 
	\begin{align}\label{weak-momen-recover}
		-	\int_0^T \int_{\Omega_t} \big( \rho \uvect \cdot \partial_t  \boldvphi   +  \rho[\uvect \otimes \uvect] : \nabla_x  \boldvphi + p_{\delta}(\rho, \theta) \Div_x \boldvphi   \big)   
		+\int_0^T \int_{\Omega_t} \Ss : \nabla_x \boldvphi 
	 \notag \\
		=
		\int_{\Omega_0}   (\rho \uvect)_{0,\delta}\cdot  \boldvphi(0, \cdot) -\int_0^T \int_{\B\setminus \Omega_t} \Ss_\omega : \nabla_x \boldvphi  + \int_0^T \int_{\B\setminus \Omega_t} \frac{a_\xi}{3}\theta^4 \Div_x \boldvphi 
	\end{align}
for any test function $\boldvphi$ satisfying \eqref{test-func}.

\vspace*{.2cm}
\noindent 
(iii)     The weak formulation for the species mass fraction in the reference domain looks like 
\begin{align}\label{eq-mass-species-solid}
	-\int_0^T \int_{\Omega_t} \big[\rho Y_k \partial_t \vphi + \rho Y_k \uvect \cdot \nabla_x \vphi  - \zeta(\theta) \nabla_x Y_k \cdot \nabla_x \vphi   \big] +  \int_0^T \int_{\B \setminus \Omega_t} \zeta_\omega(\theta) \nabla_x Y_k \cdot \nabla_x \vphi  = \notag  \\
	\int_0^T \int_{\Omega_t}  \rho \sigma_k \vphi + \int_{\Omega_0} \rho_{0,\delta} Y_{k,0, \delta} \vphi(0,\cdot)  \quad \text{ for  } k=1, ..., n  ,
	\end{align}
to be satisfied for any test function $\vphi \in \C^1([0,T)\times \B   ;\mathbb R)$,  $\vphi\geq 0$,  together with  
\begin{align}\label{eq-mass-specs-solid-2} 
-\int_0^T \partial_t  \psi \int_{\Omega_t}  \rho G(\Yvect) + & \int_0^T \psi \int_{\Omega_t} \sum_{k=1}^n\zeta(\theta) G_0 |\nabla_x Y_k|^2 
  + \int_0^T \psi \int_{\B \setminus \Omega_t} \sum_{k=1}^n\zeta_\omega(\theta) G_0 |\nabla_x Y_k|^2 
		\notag \\
		& \qquad 
		\leq  
		\int_0^T \psi \int_{\Omega_t}  \sum_{k=1}^n \rho \frac{\partial G(\Yvect)}{\partial Y_k} \sigma_k + \int_{\Omega_0} \rho_{0,\delta} G(\Yvect_{0,\delta}) \psi(0) 
	\end{align}
for any $\psi\in \C^1_c([0,T))$ with $\psi\geq 0$, $\partial_t\psi\leq 0$ and any convex function 
$G\in \C^2(\mathbb R^n;\mathbb R)$ verifying \eqref{convex-func}.

Since the quantity 
 $\int_0^T \int_{\B\setminus \Omega_t} \sum_{k=1}^n \zeta_\omega(\theta) G_0 |\nabla_x Y_k|^2$ is positive in the left-hand side of 
\eqref{eq-mass-specs-solid-2}, one can ignore this term in the formulation. 


\vspace*{.2cm}
\noindent 
(iv) Next, we shall focus on the entropy inequality (recall   that $s_\xi(\rho, \theta)=s_M(\rho, \theta)+ \frac{4a_\xi}{3\rho}\theta^3 
+\sum_{k=1}^n s_k Y_k$)
\begin{align}\label{enropy-ineq-recover}
		-\int_0^T \int_{\Omega_t} \rho s(\rho,\theta, \Yvect) \left(\partial_t \vphi + \
		\uvect \cdot \nabla_x \vphi \right) -\int_0^T \int_{\B\setminus \Omega_t} \frac{4a_\xi}{3}\theta^3  \left(\partial_t \vphi + \
			\uvect \cdot \nabla_x \vphi \right) \notag  \\
		+ \int_0^T \int_{\Omega_t} \frac{\kappa(\theta, t,x)}{\theta} \nabla_x \theta \cdot \nabla_x \vphi 
		+  \int_0^T \int_{\B \setminus \Omega_t} \frac{\kappa_\nu(\theta, t,x)}{\theta} \nabla_x \theta \cdot \nabla_x \vphi 
		\notag \\
		- \int_{\Omega_0} \rho_{0,\delta} s(\rho_{0,\delta}, \theta_{0,\delta}) \vphi(0,\cdot) 	- \int_{\B\setminus \Omega_0} \frac{4a_\xi}{3}\theta_{0,\delta}^3 \vphi(0,\cdot)
	+	 \int_0^T \int_\B \lambda \theta^4 \vphi 
		\notag \\
		\geq \int_0^T \int_{\Omega_t} \frac{\vphi}{\theta}  \left( \Ss: \nabla_x \uvect 
    + \frac{\kappa(\theta, t,x)}{\theta} |\nabla_x \theta|^2  -\sum_{k=1}^n \rho(h_k - \theta s_k) \sigma_k\right)\notag \\
		 + \int_0^T \int_{\B\setminus \Omega_t} \frac{\vphi}{\theta}  \left( \Ss_\omega : \nabla_x \uvect   + \frac{\kappa_\nu(\theta, t,x)}{\theta} |\nabla_x \theta|^2   \right)
	\end{align}
for any test function $\vphi \in \C^1_c([0,T)\times \B)$ with $\vphi\geq 0$.

\vspace*{.2cm}
\noindent 
(v)  We shall now look at the energy inequality. 
It has the following form now (recall $e_{R,\xi}(\rho,\theta)$ from \eqref{energy-pena}):
\begin{align}\label{limit-energy-pena-recover}
&	-	\int_0^T \partial_t \psi \int_{\Omega_t}  \left(\frac{1}{2} \rho |\uvect|^2  + \rho e_M(\rho,\theta)  + \rho e_{R}(\rho, \theta)  + \frac{\delta}{\beta-1} \rho^\beta \right)  \notag \\
	&	- \int_0^T \partial_t \psi \int_{\B\setminus \Omega_t} a_\xi \theta^4     + \int_0^T \int_\B \lambda \theta^5 \psi	
		 \notag  \\
	 \leq &	-\int_0^T  \psi \int_{\Omega_t}  \big(\rho[\uvect \otimes \uvect]    : \nabla_x \V  - \Ss : \nabla_x \V  + p_{\delta}(\rho,\theta)\, \Div_x \V  \big) 
		+ \int_0^T \int_{\Omega_t} \partial_t(\rho \uvect) \cdot \V  \psi  
		\notag \\ 
 &  -\int_0^T \psi \int_{\Omega_t} \sum_{k=1}^n \rho h_k \sigma_k +\psi(0) \int_{\Omega_0} 	  \left(\frac{1}{2} \frac{|(\rho \uvect)_{0,\delta}|^2}{\rho_{0,\delta}}  + \rho_{0,\delta} \left(e_M + e_{R}\right)(\rho_{0,\delta},\theta_{0,\delta})  + \frac{\delta}{\beta-1} \rho_{0,\delta}^\beta \right) \notag \\
 &  + \int_0^T \psi \int_{\B \setminus \Omega_t} \Ss_\omega : \nabla_x \V -  \int_0^T \psi \int_{\B \setminus \Omega_t} \frac{a_\xi}{3}\theta^4 \Div_x \V  + \psi(0) \int_{\B \setminus \Omega_0} a_\xi \theta^4_{0,\delta}  
	\end{align}
for all $\psi \in \C^1_c([0,T))$ such that $\psi\geq 0$ and  $\partial_t \psi \leq 0$.

\begin{remark}\label{Remark-positiveness}
Note that the integral $\int_0^T \int_{\B\setminus \Omega_t} \frac{\vphi}{\theta}  \left( \Ss_\omega : \nabla_x \uvect     + \frac{\kappa_\nu(\theta)}{\theta} |\nabla_x \theta|^2  \right)$ is simply omitted from the entropy inequality \eqref{enropy-ineq-recover}, since  this term is  positive in  $((0,T)\times \B) \setminus Q_T$.    
\end{remark}


In the next subsections, we shall find the limiting behaviors of  the weak formulations w.r.t. the parameters $\omega, \nu, \xi$ and $\lambda$. To this end, we introduce the following scaling:
\begin{align}\label{scaling} 
\lambda = \omega^{\frac{1}{2}}   = \nu^{\frac{1}{3}} =  \xi^{\frac{1}{10}} = h \quad \text{for } h>0.
\end{align}

\subsection{Bounds of the integrals in  $((0,T)\times \B) \setminus Q_T$} \label{Section-step-1} 

In this step, we shall find suitable  bounds of the integrals in 
$((0,T)\times \B) \setminus Q_T$.

\vspace*{.1cm} 
\noindent 
$\bullet$  Recall the weak formulation \eqref{weak-momen-recover} and focus on the integrals on $\B\setminus \Omega_t$. First, we see
\begin{align}\label{limit-1}  
\left|	\int_0^T \int_{\B\setminus \Omega_t} \frac{a_\xi}{3} \theta^4 \Div_x \boldvphi \right| 
&\leq C\|\Div_x \boldvphi \|_{L^\infty((0,T)\times \B)}  \frac{\xi}{\lambda^{\frac{4}{5}}} \left( \int_0^T \int_{\B} \lambda \theta^5\right)^{\frac{4}{5}}  \notag \\
& \leq    \frac{C\xi}{\lambda^{\frac{8}{5}}} 
 = :A_1(\xi,  \lambda), 
\end{align}
thanks to the bound \eqref{uniform-bound-4}. 
Using the scaling 
\eqref{scaling} we the have 
\begin{align}\label{limit-h-1}
	A_1(\xi, \lambda) =  \frac{C h^{10}}{h^{\frac{8}{5}}} .
\end{align}

Secondly, we compute 
\begin{align}\label{limit-2}
&\left|\int_0^T \int_{\B \setminus \Omega_t} \Ss_\omega : \nabla_x \boldvphi  \right|   \notag  \\
\leq & \|\nabla_x\boldvphi\|_{L^\infty((0,T)\times \B)} \left(\int_0^T \int_{\B\setminus \Omega_t} \bigg(\frac{1}{\sqrt{\theta}} |\Ss_\omega|\bigg)^{\frac{10}{9}}\right)^{\frac{9}{10}} \left( \int_0^T \int_{\B\setminus \Omega_t} \theta^5 \right)^{\frac{1}{10}}    \notag \\
 \leq &\frac{C}{\lambda^{\frac{1}{10}}} \left( \int_0^T \int_{\B\setminus \Omega_t}\lambda \theta^5 \right)^{\frac{1}{10}} 
 \left(\int_0^T \int_{\B\setminus \Omega_t} \bigg(\frac{1}{\sqrt{\theta}} \sqrt{|\Ss_\omega: \nabla_x \uvect|}\,\sqrt{|f_\omega|(1+\theta)}\bigg)^{\frac{10}{9}}\right)^{\frac{9}{10}}  \notag \\
  \leq &\frac{C}{\lambda^{\frac{1}{10}}} \left( \int_0^T \int_{\B\setminus \Omega_t}\lambda \theta^5 \right)^{\frac{1}{10}} 
 \left(\int_0^T \int_{\B\setminus \Omega_t} \frac{1}{\theta} |\Ss_\omega : \nabla_x \uvect|  \right)^{\frac{1}{2}} \left( \int_0^T \int_{\B \setminus \Omega_t} \big(f_\omega (1+\theta)\big)^{\frac{5}{4}} \right)^{\frac{2}{5}}  \notag \\
 \leq &\frac{C}{\lambda^{\frac{1}{10}}} \left( \int_0^T \int_{\B\setminus \Omega_t}\lambda \theta^5 \right)^{\frac{1}{10}}   \left(\int_0^T \int_{\B\setminus \Omega_t} \frac{1}{\theta} |\Ss_\omega : \nabla_x \uvect|  \right)^{\frac{1}{2}}
 \notag \\
 & \qquad \times \|f_\omega\|^{\frac{1}{2}}_{L^{\frac{5}{3}}(((0,T)\times \B)\setminus Q_T)}\Bigg[ 1+ \frac{1}{\lambda^{\frac{1}{10}}} \bigg(\int_0^T \int_{\B\setminus \Omega_t} \lambda \theta^5 \bigg)^{\frac{1}{10}}   \Bigg] \notag \\
 \leq &\frac{C\sqrt{\omega}}{\lambda^{\frac{1}{10}}} \frac{1}{\lambda^{\frac{1}{10} + \frac{1}{2}}}  \left[ 1+  \frac{1}{\lambda^{\frac{1}{5}}} \right]
  \leq  \frac{C\sqrt{\omega}}{\lambda^{\frac{9}{10}}}
 =:  A_2(\omega, \lambda).
 \end{align}
Here, we have used  the  fact that (thanks to the definitions \eqref{def-mu-omega}, \eqref{def-nu-omega} and the hypothesis \eqref{hypo-mu})
\begin{align*}
	\left|\Ss_\omega\right|^2 \leq  \left|\Ss_\omega : \nabla_x \uvect\right| \left|(\mu_\omega(\theta) + \eta_\omega(\theta) )\right| \leq C \left|\Ss_\omega : \nabla_x \uvect\right| \left|f_\omega\right| (1+\theta),
\end{align*}
and the bounds \eqref{uniform-bound-4} and  \eqref{uniformbound-Stress}. Also, we have used that $\|f_\omega\|_{L^{\frac{5}{3}}((0,T)\times \B)}\leq c\, \omega$ by means of the choice  \eqref{choice-f-omega}.

By means of the scaling \eqref{scaling}, we then obtain 
\begin{align}\label{limit-h-2}
	A_2(\omega, \lambda) = \frac{C h }{h^{\frac{9}{10}}}  .
\end{align}

\noindent 
$\bullet$ We now find the following estimate (appearing in \eqref{eq-mass-species-solid}
:
\begin{align}\label{bound-zeta-Y} 
\bigg| \int_0^T \int_{\B \setminus \Omega_t} \zeta_\omega(\theta) \nabla_x Y_k \cdot \nabla_x \vphi \bigg|  
&\leq C \|\nabla_x \vphi\|_{L^\infty((0,T)\times \B   )} \bigg| \int_0^T \int_{\B\setminus \Omega_t}\zeta_\omega(\theta) \nabla_x Y_k 
\bigg| \notag  \\
&\leq C \bigg(\int_0^T \int_{\B\setminus \Omega_t} \zeta_\omega(\theta) |\nabla_x Y_k|^2 \bigg)^{\frac{1}{2}} \bigg(\int_0^T \int_{\B\setminus \Omega_t} f_\omega(1+\theta) \bigg)^{\frac{1}{2}} \notag \\
&\leq C \|f_\omega\|^{\frac{1}{2}}_{L^{\frac{5}{4}}(((0,T)\times \B) \setminus Q_T) } \left[1+ \frac{1}{\lambda^{\frac{1}{10}}} \bigg( \int_0^T \int_{\B\setminus \Omega_t} \lambda \theta^5 \bigg)^{\frac{1}{10}} \right]\notag   \\
& \leq 
{C\sqrt{\omega}} \left[1+ \frac{1}{\lambda^{\frac{1}{5}}} \right] \leq \frac{C\sqrt{\omega}}{\lambda^{\frac{1}{5}}} := A_3(\omega, \lambda) ,
\end{align} 
by using the definition of $\zeta_\omega$ given in \eqref{hypo-zeta}--\eqref{def-sigma-omega}, the assumption on $f_\omega$ from \eqref{choice-f-omega} and  the bound \eqref{regular-bound-Y}. 

 The scaling \eqref{scaling} then helps to deduce 
\begin{align}\label{limit-h-zeta}
A_3(\omega, \lambda) = \frac{C h}{h^{\frac{1}{5}}} . 
\end{align}

\noindent 
$\bullet$ Next, we look to the entropy inequality \eqref{enropy-ineq-recover}. We observe that
\begin{align}\label{limit-3}
&\left|	\int_0^T \int_{\B \setminus \Omega_t} \frac{4a_\xi}{3}\theta^3\left( \partial_t \vphi +  \uvect\cdot \nabla_x \vphi \right) \right| \notag \\
\leq & \frac{C\xi}{\lambda^{\frac{3}{5}}} \|\partial_t \vphi \|_{L^\infty((0,T)\times \B)} \bigg(\int_0^T \int_{\B \setminus \Omega_t} \lambda \theta^5 \bigg)^{\frac{3}{5}} + C\|\nabla_x \vphi\|_{L^\infty((0,T)\times \B)} \bigg|\int_0^T \int_{\B\setminus \Omega_t} a \xi \theta^3 \uvect  \, \bigg| \notag \\
\leq & \frac{C\xi}{\lambda^{\frac{3}{5}}} \bigg(\int_0^T \int_{\B \setminus \Omega_t} \lambda \theta^5 \bigg)^{\frac{3}{5}}  
+ C \xi^{\frac{1}{4}}\int_0^T \bigg(\int_{\B\setminus \Omega_t} a\xi \theta^4 \bigg)^{\frac{3}{4}} \bigg(  \int_{\B\setminus \Omega_t} |\uvect|^{4} \bigg)^{\frac{1}{4}} \notag \\
\leq & \frac{C\xi}{\lambda^{\frac{3}{5}}} \times \frac{C}{\lambda^{\frac{3}{5}}}  
+ C\xi^{\frac{1}{4}} \|a_\xi \theta^4\|^{\frac{3}{4}}_{L^\infty(0,T; L^1(\B))} \| \uvect\|_{L^2(0,T; L^4(\B ; \mathbb R^3)) } \notag \\
\leq & \frac{C\xi}{\lambda^{\frac{6}{5}}} 
+ \frac{C\xi^{\frac{1}{4}}}{\lambda^{\frac{3}{4}}}\times \frac{C}{\sqrt{\omega} \lambda^{\frac{1}{2}} }
\notag \\
\leq & \frac{C\xi}{\lambda^{\frac{6}{5}}} + \frac{C \xi^{\frac{1}{4}}}{\sqrt{\omega} \lambda^{\frac{5}{4}} } 
 =:  A_4(\xi, \omega, \lambda)  ,
\end{align}  
where we have used the bounds \eqref{uniform-bound-4},   \eqref{uniform-bound-5} and \eqref{bound-theta-L4}.

Now, using the fact \eqref{scaling}, we have 
\begin{equation}\label{limit-h-3}
\begin{aligned}
	A_4(\xi, \omega, \lambda) 
 = \frac{Ch^{10}}{h^{\frac{6}{5}}} + \frac{C h^{\frac{10}{4}}}{  h^{\frac{9}{4}}  }. 
\end{aligned}
\end{equation}

\noindent 
$\bullet$ The following term in the entropy inequality \eqref{enropy-ineq-recover}  can be estimated like:
\begin{align}\label{limit-4}
&	\left|\int_0^T \int_{\B \setminus \Omega_t} \frac{\kappa_\nu(\theta,t,x)}{\theta} \nabla_x \theta \cdot \nabla_x \vphi \right| \notag \\
	\leq &\bigg(\int_0^T \int_{\B \setminus \Omega_t} \frac{\kappa_\nu(\theta,t,x)}{\theta^2} |\nabla_x \theta|^2\bigg)^{\frac{1}{2}} \bigg(\int_0^T \int_{\B \setminus \Omega_t} \kappa_\nu(\theta,t,x) |\nabla_x \vphi|^2 \bigg)^{\frac{1}{2}} \notag \\
 \leq &  \frac{C}{\lambda^{\frac{1}{2}}} \|\nabla_x \vphi\|_{L^\infty((0,T)\times \B)} 
 \sqrt{\nu} \bigg(\int_0^T \int_{\B\setminus \Omega_t}  (1+ \theta + \theta^3)  \bigg)^{\frac{1}{2}} \notag \\
  \leq & \frac{C \sqrt{\nu}}{\lambda^{\frac{1}{2}}} 
\left[1+ \frac{1}{\lambda^{\frac{1}{10}}} \bigg(\int_0^T \int_{\B\setminus \Omega_t} \lambda \theta^5 \bigg)^{\frac{1}{10}} 
+ \frac{1}{\lambda^{\frac{3}{10}}} \bigg(\int_0^T \int_{\B\setminus \Omega_t} \lambda \theta^5 \bigg)^{\frac{3}{10}} 
\right] \notag \\
 \leq & \frac{C \sqrt{\nu}}{\lambda^{\frac{1}{2}}}  \left[  1+ \frac{1}{\lambda^{\frac{1}{5}}}    + \frac{1}{\lambda^{\frac{3}{5}}}    \right] \notag \\
  \leq & \frac{C \sqrt{\nu}}{\lambda^{\frac{11}{10} } }
   = : A_5(\nu,  \lambda) ,
\end{align}
where we have used the bounds \eqref{uniform-bound-4} and \eqref{esti-temp-sub}.
We further compute (thanks to the scaling \eqref{scaling})
\begin{align}\label{limit-h-4}
A_5(\nu,  \lambda) = 
\frac{C h^{\frac{3}{2}} }{ h^{\frac{11}{10}}  }.
\end{align}

\subsection{Finding a suitable estimate for the term $\lambda \theta^4$ in the entropy balance}

 In this step, we shall find a suitable bound of the term $\disp \int_0^T \int_{\B} \lambda \theta^4 \vphi$ appearing in the entropy inequality 
\eqref{enropy-ineq-recover}, so that as $\lambda\to 0$ (in terms of the scaling), this term vanishes. 

 
 
 

  To do that, let us first write the following modified  energy inequality from \eqref{energy-pena-3-limit-veps} (having in hand that $\rho =0$ in $\B\setminus \Omega_t$ for each $t\in [0,T]$),  
  %
%
\begin{align}\label{energy-pena-recover-2}
 		& \int_{\Omega_\tau}  \left(\frac{1}{2} \rho |\uvect|^2  + \rho( e_M +  e_{R}) - \rho (s_M +s_{R})  + \frac{\delta \rho^\beta}{\beta-1}\right)(\tau, \cdot) 
+  \int_0^\tau \int_\B \lambda \theta^5  \notag  \\
 	&	
   + \int_0^\tau  \int_{\Omega_t} \frac{1}{\theta}    \left( \Ss : \nabla_x \uvect  + \frac{\kappa(\theta)|\nabla_x\theta|^2  }{\theta} - \sum_{k=1}^n \rho h_k \sigma_k \right) +  \int_0^\tau  \int_{\B\setminus\Omega_t} a_\xi \left( \theta^4 - \frac{4}{3}\theta^3  \right)   \notag \\
 \leq &	
 - \int_0^\tau   \int_{\Omega_t}  \big(\rho[\uvect \otimes \uvect]    : \nabla_x \V  - \Ss_\omega : \nabla_x \V  + p_{\xi, \delta}(\rho,\theta)\, \Div_x \V  + \rho \uvect \cdot \partial_t \V   \big) \notag \\
& + \int_0^\tau   \int_{\B\setminus \Omega_t} \Ss_\omega : \nabla_x \V  -  \int_0^\tau  \int_{\B\setminus \Omega_t} \frac{a_\xi}{3} \theta^4 \Div_x \V   
\notag \\
 		&  
   + \int_{\Omega_\tau} (\rho \uvect \cdot \V )(\tau, \cdot) 
   - \int_{\Omega_0} (\rho \uvect)_{0,\delta} \cdot \V(0, \cdot)    - \int_0^\tau \int_{\Omega_t} \sum_{k=1}^n \rho h_k \sigma_k  + \int_0^\tau  \int_\B \lambda \theta^4   
 		\notag \\ 
 		& +  \int_{\Omega_0} 	  \left(\frac{1}{2} \frac{|(\rho \uvect)_{0,\delta}|^2}{\rho_{0,\delta}}  + \rho_{0,\delta} (e_M+ e_{R, \xi})(\rho_{0,\delta},\theta_{0,\delta}) - \rho_{0,\delta} (s_M+ s_{R, \xi})(\rho_{0,\delta},\theta_{0,\delta})   + \frac{\delta \rho_{0,\delta}^\beta}{\beta-1}     \right) \notag \\
   & +  \int_{\B\setminus \Omega_0}  a_\xi \left(\theta^4_{0,\delta} -  \frac{4}{3}\theta^3_{0,\delta}  \right)
 	\end{align}
for a.a.  $\tau\in [0,T]$.

\vspace*{.1cm}

\noindent 
$\bullet$ First, by performing a similar analysis as in Section \ref{Section-step-1}, we can bound the following terms appearing in the modified energy inequality \eqref{energy-pena-recover-2}, namely, 
\begin{align} \label{limit-5}
\left|	\int_0^T \int_{\B\setminus\Omega_t} a_\xi\left(\theta^4 - \frac{4}{3}\theta^3  \right)
\right| 
+ \left| 	\int_0^T \int_{\B \setminus \Omega_t} \frac{a_\xi}{3} \theta^4 \Div_x \V \right| \leq C \left(A_1(\xi,  \lambda) + A_4(\xi, \omega, \lambda)\right)   
\end{align}
and 
\begin{align}\label{limit-6}
	\left| \int_0^T \int_{\B \setminus \Omega_t} \Ss_\omega : \nabla_x \V    \right| 
	 \leq C A_2( \omega, \lambda) .
\end{align}

\vspace*{.1cm}

\noindent 
$\bullet$ Now, the estimations of  the terms  $\disp \int_{\Omega_t} (\rho \uvect \cdot \V)(\tau, \cdot)$, $\disp \int_0^\tau \int_{\Omega_t} \rho [\uvect \otimes \uvect]: \nabla_x \V$, $\disp \int_0^\tau \int_{\Omega_t} \rho \uvect  \cdot \partial_t \V$,   $\disp \int_0^\tau \int_\B \lambda \theta^4$,  $\disp \int_0^\tau \int_{\Omega_t}\sum_{k=1}^n \frac{1}{\theta}\rho h_k \sigma_k$ and  $\disp \int_0^\tau \int_{\Omega_t}\sum_{k=1}^n \rho h_k \sigma_k$ 
are same as previous; see \eqref{esti-1}, \eqref{esti-3}, \eqref{esti-4}, \eqref{esti-5}, \eqref{rho-h_k-sigma_k} and \eqref{rho-h-sig-thet} respectively (since all of those estimates are uniform w.r.t.  $\lambda$).

\vspace*{.1cm}

\noindent 
$\bullet$   We also recall that all the terms in the fluid domain $Q_T$ are independent of the parameters $\omega, \xi, \nu$. In
what follows, we compute that
\begin{align}\label{esti-final-1} 
	\int_0^\tau \int_{\Omega_t} \Ss : \nabla_x \V 
	&\leq   \frac{1}{2} \int_0^\tau \int_{\Omega_t} \frac{1}{\theta} \Ss : \nabla_x \uvect + C(\V) \int_0^\tau \int_{\Omega_t} \theta \notag  \\
	& \leq \frac{1}{2} \int_0^\tau \int_{\Omega_t}  \frac{1}{\theta} \Ss : \nabla_x \uvect + 
		C(\V,a) \bigg(1+\int_0^\tau \int_{\Omega_t} a \theta^4  \bigg)  \notag \\   
	& \leq \frac{1}{2} \int_0^\tau \int_{\Omega_t}  \frac{1}{\theta} \Ss : \nabla_x \uvect + 
		C(\V,a) \int_0^\tau \int_{\Omega_t} \rho e (\rho, \theta) + C(\V,a) 
\end{align}
since  $a \theta^4  \leq \rho e(\rho, \theta)$.

Note that the term
$\disp \frac{1}{2} \int_0^\tau \int_{\Omega_t}  \frac{1}{\theta} \Ss: \nabla_x \uvect$ 
can be absorbed by the associated term in the left-hand side of \eqref{energy-pena-recover-2}. 



\vspace*{.1cm}

\noindent 
$\bullet$ Let us recall that the pressure term in $Q_T$ reads as 
$\disp p_\delta(\rho, \theta) = p_{M}(\rho, \theta) + \frac{a}{3}\theta^4 + \delta \rho^\beta$. Then by means of the  bound \eqref{mole_bound_p_M},   we get
	\begin{align}\label{estimate-pressue-final}
		\left|\int_0^\tau \int_{\Omega_t}  p_{\delta} (\rho, \theta) \Div_x \V \right| 
	&	\leq C(\V) \int_0^\tau \int_{\Omega_t} \left(\frac{\delta}{\beta-1} \rho^\beta + a \theta^4 + \rho^{\frac{5}{3}} + \theta^{\frac{5}{2}} \right) \notag \\
	& \leq C(\V, p_\infty, a) \int_0^\tau\int_{\Omega_t} \left(\frac{\delta}{\beta-1} \rho^\beta + \rho e(\rho, \theta) \right).
	\end{align}

\vspace*{.1cm}
\noindent
$\bullet$ We also deduce that 
\begin{align}\label{esti-initial}
	\left| \int_{\B\setminus \Omega_0} a_\xi \Big(\theta^4_{0,\delta} - \frac{4}{3} \theta^3_{0,\delta} \Big) \right| \leq C(\theta_0) \xi .
	\end{align}

\vspace*{.2cm}

Using all the above estimates and the bounds of the terms in  $((0,T)\times\B)\setminus Q_T$ given by \eqref{limit-1}, \eqref{limit-2}, \eqref{bound-zeta-Y}, 
\eqref{limit-3}, 
\eqref{limit-4} 
\eqref{limit-5} and \eqref{limit-6}, we have from 
\eqref{energy-pena-recover-2} (also by utilizing the Gr\"onwall's lemma)
\begin{align}\label{energy-pena-recover-3}
 		& \int_{\Omega_\tau}  \left(\frac{1}{2} \rho |\uvect|^2  + \rho( e_M +  e_{R}) - \rho (s_M +s_{R})  + \frac{\delta \rho^\beta}{\beta-1}\right)(\tau, \cdot) 
+  \int_0^\tau \int_\B \lambda \theta^5  \notag \\
 	&	
   + \int_0^\tau  \int_{\Omega_t} \frac{1}{\theta}    \left( \Ss : \nabla_x \uvect  + \frac{\kappa(\theta)|\nabla_x\theta|^2  }{\theta} \right) \notag   \\
 & \leq 
	  C	\int_{\Omega_0} 	  \bigg(\frac{1}{2} \frac{|(\rho \uvect)_{0,\delta}|^2}{\rho_{0,\delta}}  + \rho_{0,\delta} (e_M+ e_{R, \xi})(\rho_{0,\delta},\theta_{0,\delta}) - \rho_{0,\delta} (s_M+ s_{R, \xi})(\rho_{0,\delta},\theta_{0,\delta}) \notag  \\
 & \qquad \qquad + \frac{\delta \rho_{0,\delta}^\beta}{\beta-1}      - (\rho \uvect)_{0,\delta} \V(0) +1 \bigg) \notag \\ 
	& \ + 
 C \xi + C \left( A_1(\xi,  \lambda)  +  A_2(\omega, \lambda) +  A_3(\omega, \lambda) +  A_4(\xi, \omega, \lambda) +A_5(\nu, \lambda) \right) 
\end{align}
for a.a.  $\tau \in [0,T]$.

\vspace*{.2cm}
\paragraph{\bf Bounds of the term $\lambda \theta^4$.} 
 Let us recall the entropy balance 
 \eqref{enropy-ineq-recover}. The only term left to estimate is the integral concerning $\lambda \theta^4$. In fact, thanks to \eqref{energy-pena-recover-3}, we have 
\begin{align}\label{bound-lambda-theta-4}
	\left|\int_0^T \int_\B \lambda \theta^4 \vphi \right| 
	&\leq  \lambda^{\frac{3}{5}} \|\vphi\|_{L^\infty((0,T)\times \B)}\bigg(\int_0^T \int_\B \lambda \theta^5 \bigg)^{\frac{4}{5}} \notag  \\
	&\leq C \lambda^{\frac{3}{5}} \Big(1+ \xi +  A_1(\xi,  \lambda)  +  A_2(\omega, \lambda) +  A_3(\omega, \lambda) +  A_4(\xi, \omega, \lambda) + A_5(\nu, \lambda) \Big)^{\frac{4}{5}},
\end{align} 
where the above constant $C$ does not depend on any of the parameters 
$\lambda$, $\xi$, $\omega$, $\nu$ or $\delta.$

\subsection{Passing to the limits of $\omega,\xi, \nu, \lambda$}\label{section-limit-pass}

Now, we are in position to pass to the limits of all the parameters $\omega$, $\xi$, $\nu$, $\lambda$ together.  
We keep in mind the scaling introduced in \eqref{scaling} w.r.t. $h$ and  recall the estimates  \eqref{limit-1}--\eqref{limit-h-1},  \eqref{limit-2}--\eqref{limit-h-2},  \eqref{bound-zeta-Y}--\eqref{limit-h-zeta},   \eqref{limit-3}--\eqref{limit-h-3} and \eqref{limit-4}--\eqref{limit-h-4}, from which it is not difficult to observe that 
\begin{align}\label{limits-all}
\begin{dcases}
		A_1(\xi, \lambda) =  \frac{C h^{10}}{h^{\frac{8}{5}}} 
 \to 0 \ \text{ as } h\to 0, \\
	A_2(\omega, \lambda) = \frac{C h}{h^{\frac{9}{10}}} 
 \to 0 \ \text{ as } h\to 0, \\
	A_3( \omega, \lambda) = \frac{Ch}{h^{\frac{1}{5}}} \to 0  \  \text{ as } h\to 0, 
	\\
A_4( \xi, \omega, \lambda) = \frac{Ch^{10}}{h^{\frac{6}{5}}} + \frac{Ch^{\frac{10}{4}}}{h^{\frac{9}{4}}} 
	\to 0 \  \text{ as } h \to 0, \text{ and}
	\\
A_5(\nu ,\lambda) = \frac{C h^{\frac{3}{2}}}{h^{\frac{11}{10}}} 
\to 0 \ \  \text{ as } h\to 0.
\end{dcases}
\end{align}

Next, by means of the estimate 
\eqref{limit-1}--\eqref{limit-h-1}, one can deduce 
\begin{align*}
&\int_0^T \partial_t \psi \int_{\B\setminus\Omega_t} a_\xi \theta^4 , \ \ \int_0^T \psi \int_{\B\setminus \Omega_t}  \frac{a_\xi}{3}\theta^4 \Div_x \V   \to 0 \ \ \text{ as } h \to 0, 
\end{align*}
and analogously, in terms of \eqref{limit-6},  one can say   
\begin{align*}
\int_0^T \int_{\B \setminus \Omega_t} \Ss_\omega : \nabla_x \V \psi  \to 0  \ \ \text{ as } h \to 0, 
\end{align*}
in the energy inequality \eqref{limit-energy-pena-recover}.

Furthermore, one can easily observe that  
\begin{align*}
\psi(0) \int_{\B \setminus \Omega_0} a_\xi \theta^4_{0,\delta} \ \text{ and } \
\int_{\B\setminus \Omega_0} \frac{4a_\xi}{3} \theta^3_{0,\delta} \vphi(0, \cdot) \to 0 \ \ \text{ as } \xi \to 0,
\end{align*}
respectively in the energy inequality \eqref{limit-energy-pena-recover} and  entropy inequality \eqref{enropy-ineq-recover}.

Now, thanks to the estimate \eqref{bound-lambda-theta-4}, and the  limiting behaviors of $A_1, A_2, A_3, A_4$ and $A_5$ from \eqref{limits-all},
we have 
\begin{align}\label{limit-h-5}
\int_0^T \int_\B \lambda \theta^{4} \vphi \to 0 \quad \text{as } h\to 0 ,
\end{align}
in the entropy inequality  \eqref{enropy-ineq-recover}.

Finally, the term  
$\displaystyle \int_0^T \int_{\B} \lambda\theta^{5}  \psi$ in the energy inequality \eqref{limit-energy-pena-recover} can be omitted since it is non-negative.

\subsection{Weak formulations after passing to the limits}

Let us write the weak formulations 
after passing to the limit of $\xi$, $\omega$, $\nu$, $\lambda$ to $0$ (using the limiting behavior discussed in Section \ref{section-limit-pass}).

\vspace*{.1cm}
\noindent 
(i)  The  weak formulation of the continuity equation will be the same as \eqref{weak-conti-recover} after passing to the limits.

\vspace*{.1cm}

\noindent
(ii) The weak formulation of the momentum equation becomes (from \eqref{weak-momen-recover})
	\begin{align}\label{weak-momen-recover-new}
		-	\int_0^T \int_{\Omega_t} \big( \rho \uvect \cdot \partial_t  \boldvphi   +  \rho[\uvect \otimes \uvect] : \nabla_x  \boldvphi + p_{\delta}(\rho, \theta) \Div_x \boldvphi   \big)   
		+\int_0^T \int_{\Omega_t} \Ss : \nabla_x \boldvphi 
	\notag  \\
		=
		\int_{\Omega_0}   (\rho \uvect)_{0,\delta}\cdot  \boldvphi(0, \cdot)
	\end{align}  
for any test function $\boldvphi$  as in \eqref{test-func}.

 Next, by arguing as  \cite[Section 4.3.1]{Sarka-Kreml-Neustupa-Feireisl-Stebel}, we can show that the momentum equation \eqref{weak-momen-recover-new}, in fact, holds for any test function   
\begin{align*}
	\boldvphi \in \C^\infty_c([0,T)\times \B;\mathbb R^3) \quad \text{with} \quad \boldvphi(\tau, \cdot) \cdot \mathbf n \big|_{\Gamma_\tau} =0 \ \ \text{for any } \tau\in [0,T].
\end{align*}

\vspace*{.2cm}
\noindent 
(iii)     After passing to the limit in 
\eqref{eq-mass-species-solid},    the  weak formulation for the species mass fractions reads as  
\begin{align}\label{eq-mass-species-solid-new}
&-\int_0^T \int_{\Omega_t} \big[\rho Y_k \partial_t \vphi + \rho Y_k \uvect \cdot \nabla_x \vphi  - \zeta(\theta) \nabla_x Y_k \cdot \nabla_x \vphi   \big] 
   \notag \\
& \qquad \qquad \qquad =		\int_0^T \int_{\Omega_t}  \rho \sigma_k \vphi + \int_{\Omega_0} \rho_{0,\delta} Y_{k,0, \delta} \vphi(0,\cdot)  \quad \text{ for  } k=1, ..., n  ,
	\end{align}
to be satisfied for any test function 
 $\vphi \in \C^1([0,T)\times \B   ;\mathbb R)$,  $\vphi\geq 0$,  together with  
	\begin{align}\label{eq-mass-specs-solid-new-2} 
		-\int_0^T \partial_t  \psi \int_{\Omega_t}  \rho G(\Yvect) + & \int_0^T \psi \int_{\Omega_t} \sum_{k=1}^n\zeta(\theta) G_0 |\nabla_x Y_k|^2 
		\notag \\
		& \qquad 
		\leq  
		\int_0^T \psi \int_{\Omega_t}  \sum_{k=1}^n \rho \frac{\partial G(\Yvect)}{\partial Y_k} \sigma_k + \int_{\Omega_0} \rho_{0,\delta} G(\Yvect_{0,\delta}) \psi(0) 
	\end{align}
for any $\psi\in \C^1_c([0,T))$ with $\psi\geq 0$, $\partial_t\psi\leq 0$ and any convex function $G\in \C^2(\mathbb R^n;\mathbb R)$ verifying  \eqref{condition-ellipic-G}.

Since we have $\rho=0$ outside $\B\setminus \Omega_t$ for each $t\in [0,T]$,  the same holds true  for the species mass density $\rho_k$ due to the relation  \eqref{rho-1} for each $k=1,...,n$. As a consequence, the principle of mass conservation: $\sum_{k=1}^n Y_k=1$ remains valid in $Q_T$.   


\vspace*{.2cm}
\noindent 
(iv) The entropy inequality 
\eqref{enropy-ineq-recover} now follows:
\begin{align}\label{enropy-ineq-recover-new}
&-\int_0^T \int_{\Omega_t} \rho s(\rho,\theta, \Yvect) \left(\partial_t \vphi + 
		\uvect \cdot \nabla_x \vphi \right) 	
		+ \int_0^T \int_{\Omega_t} \frac{\kappa(\theta, t,x)}{\theta} \nabla_x \theta \cdot \nabla_x \vphi 
	 - \int_{\Omega_0} \rho_{0,\delta} s(\rho_{0,\delta}, \theta_{0,\delta}) \vphi(0,\cdot) \notag \\   
	& \qquad 	\qquad  \geq \int_0^T \int_{\Omega_t} \frac{\vphi}{\theta}  \left( \Ss: \nabla_x \uvect 
    + \frac{\kappa(\theta, t,x)}{\theta} |\nabla_x \theta|^2  -\sum_{k=1}^n \rho(h_k - \theta s_k) \sigma_k\right) 
	\end{align}
for any test function  $\vphi \in \C^1_c([0,T)\times \B)$ with $\vphi\geq 0$.

\vspace*{.2cm}
\noindent 
(v)  On the other hand, the energy inequality \eqref{limit-energy-pena-recover} is reduced to
\begin{align}\label{limit-energy-pena-recover-new}
&	-	\int_0^T \partial_t \psi \int_{\Omega_t}  \left(\frac{1}{2} \rho |\uvect|^2  + \rho e_M(\rho,\theta)  + \rho e_{R}(\rho, \theta)  + \frac{\delta}{\beta-1} \rho^\beta \right) \notag 
		  \\
	 \leq &	-\int_0^T  \psi \int_{\Omega_t}  \big(\rho[\uvect \otimes \uvect]    : \nabla_x \V  - \Ss : \nabla_x \V  + p_{\delta}(\rho,\theta)\, \Div_x \V  \big) 
		+ \int_0^T \int_{\Omega_t} \partial_t(\rho \uvect) \cdot \V  \psi  
		\notag \\ 
 &    -\int_0^T \psi \int_{\Omega_t} \sum_{k=1}^n \rho h_k \sigma_k +\psi(0) \int_{\Omega_0} 	  \left(\frac{1}{2} \frac{|(\rho \uvect)_{0,\delta}|^2}{\rho_{0,\delta}}  + \rho_{0,\delta} \left(e_M + e_{R}\right)(\rho_{0,\delta},\theta_{0,\delta})  + \frac{\delta}{\beta-1} \rho_{0,\delta}^\beta \right) 
	\end{align}
for any $\psi \in \C^1_c([0,T))$ such that $\psi\geq 0$ and $\partial_t \psi \leq 0$.


\section{Conclusion of the proof: vanishing artificial pressure}

To conclude, we need to pass to the limit $\delta\to 0$ to get rid of the artificial pressure term $\delta \rho^\beta$
and to obtain the weak formulations as stated in Section \ref{Section-weak-target}. 
 This can be done by following exactly the same  approach as developed in \cite[Section 6]{Eduard-CPAA-2008}. We also refer \cite[Section 4.4]{sarka-aneta-al-JMPA} for a sketch of the proof.  
 
 
This concludes the proof of 
\Cref{Main-theorem}.

\appendix

\section{Poincar\'{e} type inequalities} 

\begin{lemma}[Generalized Poincar\'e inequality]\label{Poincare}
	Let $1\leq q \leq \infty$, $0<\gamma<\infty$, $U_0>0$ and $\Omega\subset \mathbb R^N$ be a bounded Lipschitz domain. 

Then, there exists a positive constant $C=C(q, \gamma, U_0)$ such that 
\begin{align}
	\|v \|_{W^{1,q}(\Omega)} \leq C \left[\| \nabla_x v\|_{L^q(\Omega; \mathbb R^N)} + \Big( \int_U |v|^\gamma \Big)^{\frac{1}{\gamma}}  \right] 
\end{align}
for any measurable $U\subset \Omega$, $|U|\geq U_0$ and any $v\in W^{1,q}(\Omega)$. 
\end{lemma}

A formal proof of the above result is given in \cite[Theorem 11.20, Chapter 11.9]{Feireisl-Novotny-book}.

\begin{lemma}[Generalized Korn-Poincar\'e inequality]\label{Korn-Poincare}
Let $\Omega \subset \mathbb R^N$, $N>2$, be a bounded Lipschitz domain. Assume that $r$ is a
non-negative function such that 
\begin{align*}
	0< m_0 \leq \int_\Omega r , \quad \int_\Omega r^\gamma \leq k,
\end{align*}
for some specific $\gamma>1$. Then 
\begin{align}
	\|\mathbf v\|_{W^{1,q}(\Omega; \mathbb R^N)} \leq C(q, m_0, k) \left(\left\| \nabla_x  \mathbf v + \nabla_x^\top \mathbf v -\frac{2}{N} \Div_x \mathbf v \mathbb I   \right\|_{L^q(\Omega; \mathbb R^{N\times N})}  + \int_\Omega r |\mathbf v| \right) 
\end{align}
for any $\mathbf v\in W^{1,q}(\Omega; \mathbb R^N)$ and $1<q<\infty$.
\end{lemma}

We refer \cite[Theorem 11.23, Chapter 11.10]{Feireisl-Novotny-book} for the proof of the above lemma.

\section*{Acknowledgments}
  
 K. Bhandari, S. Kra\v{c}mar  and {\v{S}}. Ne{\v{c}}asov{\'{a}} have been supported by the Czech-Korean project GA\v{C}R/22-08633J. Moreover,  K. Bhandari  and {\v{S}}. Ne{\v{c}}asov{\'{a}}
 have been supported by  the Praemium Academiae of {\v{S}}. Ne{\v{c}}asov{\'{a}}, and S. Kra\v{c}mar is supported by  RVO: 12000.
   M. Yang has been supported by the National Research Foundation of Korea (NRF) grant funded by the Korean Government (MSIT)  (No. 2021K2A9A1A06091213).
  Finally, the Institute of Mathematics, CAS is supported by RVO:67985840.

\bibliographystyle{plain}
\bibliography{ref_NSF-Multicomp}

\end{document}